\documentclass[11pt, letterpaper]{amsart}

\calclayout

\usepackage[thinlines]{easytable}
\usepackage{graphicx}
\usepackage{amssymb}
\usepackage{amsmath}
\usepackage[dvipsnames]{xcolor}
\usepackage{commath}
\usepackage{amsfonts}%
\usepackage{graphicx,verbatim}
\usepackage{tikz}
\usetikzlibrary{cd}
\usepackage{bbm}

\usepackage[colorlinks=true]{hyperref}

\usepackage{mathtools} 

\usepackage{amsthm}

\usepackage{units}

\allowdisplaybreaks

\usepackage{epigraph}

\usepackage{dirtytalk}

\usepackage{lettrine}

\newcommand{\cO}{\mathcal{O}}

\newcommand{\D}{\mathbb{D}}

\newcommand{\depth}{\mbox{depth}}

\newtheorem{theorem}{Theorem}
\newtheorem{definition}{Definition}
\newtheorem{proposition}{Proposition}
\newtheorem{corollary}{Corollary}
\newtheorem{lemma}{Lemma}

\newtheorem{remark}{Remark}

\newtheorem{notation}{Notation}

\newcommand{\A}{\mathbb{A}}

\newcommand{\C}{\mathbb{C}}
\newcommand{\Q}{\mathbb{Q}}
\newcommand{\Z}{\mathbb{Z}}

\newcommand{\R}{\mathbb{R}}

\newcommand{\G}{\mathbb{G}}

\newcommand{\F}{\mathbb{F}}

\newcommand{\fm}{\mathfrak{m}}

\newcommand{\GL}{\mathrm{GL}}

\newcommand{\Ker}{\mathrm{Ker}}

\newcommand{\Hom}{\mathrm{Hom}}

\newcommand{\CO}{\mathcal{O}}

\newcommand{\End}{\mathrm{End}}

\newcommand{\Tr}{\mathrm{Tr}}
\newcommand{\Iel}{\mathrm{I}_{\mathrm{ell}}}

\newcommand{\Res}{\mathrm{Res}}

\newcommand{\vol}{\mathrm{vol}}

\newcommand{\Mat}{\mathrm{Mat}}

\newcommand{\Pol}{\mathrm{Pol}}
\newcommand{\diag}{\mathrm{diag}}

\newtheorem{lthm}{Theorem}

\newtheoremstyle{upright}
  {6pt}{6pt}
  {\normalfont}
  {}{\bfseries}
  {.}{.5em}{}

\theoremstyle{upright}
\newtheorem{examplex}{Example}

\usepackage{lineno}

\begin{document}





\title[]{Beyond Endoscopy for $\GL(3, \mathbb{Q})$: Isolation of the residual spectrum via Poisson Summation}
\author{}
\address{}
\email{}

\author{Taiwang  Deng}
\address[T.D.]{Yanqi Lake Beijing Institute of Mathematical Sciences and Applications (BIMSA), Huairou District, 100084, Beijing\\
China}
\email{dengtaiw@bimsa.cn}

\author{Malors Espinosa}
\address[M.E.]{Yau Mathematical Sciences Center, Tsinghua University, Haidian District, Beijing, 100084, China.}
\email{espino41@mail.tsinghua.edu.cn}

\date{}


\begin{abstract}
We extend to \(\GL(3,\Q)\) the Poisson-summation method developed by
Altu\u{g} for \(\GL(2,\Q)\) in Beyond Endoscopy.  For a fixed prime and
a family of factorizable test functions, we isolate the contribution of
the trivial representation from the regular elliptic part of the trace
formula and obtain an explicit expansion of
\[
\Iel(f^{p,k})-\Tr(\mathbf{1}(f^{p,k})).
\]
We rewrite the finite orbital integrals in terms of an overorder zeta
function attached to the monogenic cubic orders defined by the
characteristic polynomials.  Its identification with Yun's zeta function
supplies the functional equation needed for an approximate functional
equation.  A periodicity theorem then
makes Poisson summation in the two polynomial coefficients possible.  The
zero frequency is governed by Kloosterman-type sums and a  Dirichlet
series that admits a uniform evaluation.  The residues at the origin recover
the trace of the trivial representation and one third of the trace of the
normalized representation induced from the trivial character of the Borel
subgroup.  
\end{abstract}

\maketitle


\tableofcontents
\section{Introduction}\label{sec:introduction}

\subsection{Background}

Beyond Endoscopy is Langlands' trace-formula approach to the Principle of
Functoriality \cite{LanBE04,Langlangs13}.  Its aim is
to construct distributions that detect automorphic representations through
the poles of their \(L\)-functions.  Arthur subsequently placed the proposal
in a broader trace-formula framework
\cite{Arthur2018,ARTHUR2018425, Arthur2017}.  The method is concrete only
after several analytic and arithmetic difficulties have been resolved,
including the removal of nontempered residual contributions.

Altu\u{g} carried out this isolation for \(\GL(2,\Q)\)
\cite{AliI,AliII,AliIII}.  His starting point is the regular elliptic
distribution
\[
\Iel(\varphi)=\sum_{[\gamma]}\vol(\gamma)\cO(\gamma,\varphi).
\]
The characteristic polynomial of \(\gamma\) supplies an integral parameter,
and a suitable zeta function rewrites the finite orbital integrals as a
Dirichlet series with a functional equation.  After an approximate
functional equation, the elliptic distribution becomes amenable to Poisson
summation in that parameter.  The zero frequency and a subsequent contour
shift then recover the trace of the trivial representation.  Variants of
this method over number fields and with additional ramification appear in
\cite{BEgenfields,Cheng26}; a complementary Kuznetsov-formula approach is
developed in \cite{Venk02}.

The passage from \(\GL_2\) to \(\GL_3\) introduces two essential
complications.  For the test functions considered here, regular elliptic
classes are parametrized by the two-variable family
\[
X^3-aX^2+bX\mp p^k,
\]
and the monogenic orders generated by their roots need not be maximal.
Although these monogenic orders are Gorenstein, their overorders need not
be.  Moreover, the coefficients of the relevant zeta functions are
indexed by the entire overorder lattice rather than by a single conductor.
Their dependence on \((a,b)\) is therefore not visible from a reciprocity
symbol, and the periodicity required for Poisson summation becomes a
problem in the arithmetic of cubic rings.

Yun's zeta function provides the conceptual bridge between orbital
integrals and orders.  At each finite place, for the unit function
\(\mathbf 1_{\Mat_3(\Z_q)}\), its normalized local factor satisfies
\[
\widetilde J_{\gamma,q}(0)=\widetilde J_{\gamma,q}(1)
=\operatorname{Orb}(\mathbf 1_{\Mat_3(\Z_q)},\gamma),
\]
while the global completion satisfies a functional equation
\cite[Corollary~4.6 and Theorem~1.2]{ZYun}.
Motivated by this construction and by the quadratic Zagier zeta function,
we associate to a cubic order \(R_0\) an explicit overorder zeta function
\[
L(s,R_0)=
\sum_{R_0\subset R\subset\mathcal O_E}
h_R(s)L_R(s)[R:R_0]^{1-2s}.
\]
Its coefficients are sufficiently explicit for the lattice completion and
Poisson summation used below.  The identity
\begin{equation}
J_R(s)=L(s,R)\zeta_\Q(s)
\end{equation}
for locally monogenic cubic orders is now a theorem
\cite[Corollary~9.1]{DengEspinosaYunZeta}; see
Theorem~\ref{teo: identification-to-Yun}.  Thus the functional equation
used in the present paper is unconditional.

\subsection{Main results}

Fix a prime \(p\), an integer \(k\geq0\), a real number
\(0<\alpha<1\), and the factorizable test function \(f^{p,k}\) of
Definition~\ref{def: test function}.  Our principal result
isolates the trace of the trivial representation from the regular elliptic
part of the trace formula for \(\GL(3,\Q)\).

\begin{lthm}[Theorem \ref{thm: trace is there}]
\label{lthm: main theorem}
One has
\begin{align*}
\Iel(f^{p,k})-\Tr(\mathbf 1(f^{p,k}))
={}&\frac13\Tr(\xi(f^{p,k}))
+\mathcal C_{3,\alpha}(f^{p,k})\\
&+\mathcal Q_1(f^{p,k})+\mathcal Q_2(f^{p,k})
+\Sigma((\xi,\eta)\ne(0,0))-\Sigma(f^{p,k}).
\end{align*}
\end{lthm}

Here \(\xi\) is the normalized representation induced from the trivial
character of the Borel subgroup.  The term
\(\mathcal C_{3,\alpha}\), defined in
\eqref{eq:combined-cubic-residue-definition}, is the sum of the
zero-frequency residues at \(z=-2/3\) in \(R_1\) and at \(u=1/3\) in
\(R_2\).  The preceding residues, at \(z=-1/2\) and \(u=1/2\), cancel by
Lemma~\ref{lem:midpoint-cancellation-GL3}.
The term \(\Sigma(f^{p,k})\) is the completion needed before Poisson
summation, and \(\Sigma((\xi,\eta)\ne(0,0))\) denotes the contribution of
the nonzero dual frequencies.  Estimates for these terms and the two
remaining contour integrals form analytic problems for future work.

Two arithmetic results drive the proof.  First, for
\(\epsilon\in\{\pm1\}\) and every \((a,b)\in\Z^2\), put
\[
R^\epsilon(a,b):=
\Z[X]/(X^3-aX^2+bX-\epsilon p^k).
\]
This is a finite free cubic ring even when its discriminant is zero; on
\(V(\epsilon)\) it is the original monogenic order.  For fixed
\(n,f\geq1\), put
\[
S^\pm(a,b;n,f):=
\sum_{\substack{R\supseteq R^\pm(a,b)\\{}[R:R^\pm(a,b)]=f}}
\sum_{d\mid n}d\,a_R(d)W_R(n/d)
.\]
Here \(R\) runs over finite-index cubic over-rings in
\(R^\pm(a,b)\otimes\Q\).  Then \(S^\pm(a,b;n,f)\) depends only on
\((a,b)\bmod nf^2\); see Theorem~\ref{thm: periodicity}.  Thus the same
ring-theoretic definition applies to every residue class, including
classes represented by a polynomial of discriminant zero.

Second, define the Kloosterman-type sums
\[
K^\pm(n,f):=
\sum_{(a,b)\,(\mathrm{mod}\,nf^2)}
S^\pm(a,b;n,f)
.\]
and the associated  Dirichlet series
\[
\D^\pm(z):=\sum_{n,f\geq1}
\frac{K^\pm(n,f)}{n^{z+2}f^{2z+3}}.
\]
Theorem~\ref{thm: Dirichlet Sums Evaluation} gives, for
\(\Re(z)>1\),
\[
\D^\pm(z)=
\frac{\zeta(2z)}{\zeta(z+1)}
\frac{\zeta(3z)}{\zeta(2z+1)}
\frac{1-p^{-z(k+1)}}{1-p^{-z}}
\frac{1-p^{-z(k+2)}}{1-p^{-2z}}.
\]
This expression is independent of the sign and gives the meromorphic
continuation of \(\D^\pm\).  Substitution into the two Mellin integrals
makes the residues explicit.  At the origin the first integral gives
\(\Tr(\mathbf 1(f^{p,k}))\), and the second gives
\(\frac13\Tr(\xi(f^{p,k}))\).  Before reaching the origin, the second
integral also crosses the simple poles of \(\zeta(2u)\) and
\(\zeta(3u)\) at \(u=1/2\) and \(u=1/3\).  The first of these residues
is canceled by the pole of \(\zeta(2z+2)\) at \(z=-1/2\) in the first
Mellin integral.  Moving the first contour further crosses the pole of
\(\zeta(3z+3)\) at \(z=-2/3\).  This residue and the \(u=1/3\)
residue are the two terms comprising \(\mathcal C_{3,\alpha}\).

\subsection{Previous work}

For \(\GL_2\), Langlands' divisor-sum formula rewrites the finite orbital
integrals in terms of quadratic orders \cite{LanBE04}.  Altu\u{g}
recognized its analytic continuation as Zagier's zeta function and used
its functional equation and explicit Dirichlet coefficients in Poisson
summation \cite{AliI,zagier}.  Arthur later proposed Yun's zeta function
as the intrinsic higher-rank replacement
\cite{ARTHUR2018425} while the quadratic comparison was proved in
\cite{malors21-2,malors21}.  The arithmetic difficulties in degree three
are more substantial because cubic overorders need not form a chain and
may be non-Gorenstein.

An earlier version of this paper formulated the equality
\(J_R(s)=L(s,R)\zeta_\Q(s)\) as a conjecture and used it to obtain the
functional equation of the overorder zeta function.  Lee subsequently
gave the first unconditional proof of the functional equation for the
completed overorder zeta function \(\Lambda(s,R)\)---equivalently, for
the naturally completed product \(L(s,R)\zeta_\Q(s)\)---for Gorenstein
orders in cubic number fields \cite[Theorem~1.3]{LeeCubicFE}.  The
zeta-function identity itself was then proved for locally monogenic cubic
orders in \cite{DengEspinosaYunZeta}, by an argument independent of Lee's
local-factor computation.  That work also gives a short uniform
evaluation of the local Kloosterman series, including residue
characteristics \(2\) and \(3\), replacing the earlier case-by-case
calculation.  The present version incorporates these results, so the
isolation theorem no longer depends on a conjecture.

The role of the trivial representation is dictated by the residual
spectrum.  It is the unique nontempered residual contribution in the
present setup and must be removed before the prime or integer averages
predicted by Beyond Endoscopy can be expected to converge
\cite{frenkel2010formule}.  The distinction between prime averages and
integer averages, emphasized by Sarnak \cite{Sarnak01}, becomes important
at that later stage.  Here we work at one fixed prime, as in the first
stage of Altu\u{g}'s argument.

\subsection{Sketch of the proof}

The proof follows five steps.

\emph{1. Rewriting the elliptic distribution.}
The support of the test function forces the characteristic polynomial of
a contributing class to be \(X^3-aX^2+bX\mp p^k\).  Its root generates a
cubic order \(R(a,b)\).  The volume term is expressed by the class-number
formula, and the product of finite orbital integrals is replaced by
\(L(1,R(a,b))\).  The archimedean orbital integral becomes a smooth
function \(\theta^\pm\) of the normalized coefficients on the regular
coefficient locus.

\emph{2. Approximate functional equation.}
The identity with Yun's zeta function gives the functional equation
\(\Lambda(s,R)=\Lambda(1-s,R)\).  Applying an approximate functional
equation and expanding \(h_R\) and \(L_R\) produces sums involving
\(a_R(d)\), \(W_R(n)\), and the overorders of \(R(a,b)\), multiplied by
smooth weights \(\Phi^\pm_{n,f}\) and \(\Psi^\pm_{n,f}\).

\emph{3. Completion and Poisson summation.}
For every \((a,b)\in\Z^2\), including the discriminant-zero locus, we
attach the finite free cubic ring
\(\Z[X]/(X^3-aX^2+bX\mp p^k)\) and form the same finite-index
over-ring sum that occurs on the regular elliptic locus.  For fixed
\(n\) and \(f\), Theorem~\ref{thm: periodicity} shows directly that this
arithmetic coefficient is periodic modulo \(nf^2\).
On the analytic side, Theorem~\ref{thm: Schwartz Function} shows that
\(\Phi^\pm_{N,f}\) and \(\Psi^\pm_{N,f}\) extend smoothly by zero across
the discriminant-zero locus.  Enlarging the summation from
\(V(\pm)\) to \(\Z^2\) therefore gives
\(\Iel(f^{p,k})+\Sigma(f^{p,k})\).  At a
zero-discriminant point the cubic ring and its arithmetic coefficient
are defined, while the full summand vanishes because its analytic
weight is zero.  Poisson summation on each residue class modulo
\(nf^2\) then yields a zero-frequency term and the remainder
\(\Sigma((\xi,\eta)\ne(0,0))\).

\emph{4. Kloosterman Dirichlet series.}
The zero frequency is governed by \(K^\pm(n,f)\).  These sums factor
locally, and their Euler factors are evaluated uniformly in
\cite[Section~10]{DengEspinosaYunZeta}.  Multiplication of the local
factors gives the formula for \(\D^\pm(z)\) above.

\emph{5. Residues.}
The zero-frequency contribution splits into two Mellin integrals
\(R_1\) and \(R_2\).  The \(R_1\)-contour may be moved to
\(-5/6<\Re(z)<-2/3\), crossing \(z=0\), \(z=-1/2\), and
\(z=-2/3\).  The first residue is
\(\Tr(\mathbf1(f^{p,k}))\).  Moving the \(R_2\)-contour to
\(1/3<\Re(u)<1/2\) crosses \(u=1/2\); its residue cancels the
\(z=-1/2\) residue exactly.  An indented contour then encloses
\(u=1/3\) and \(u=0\).  The \(u=1/3\) residue combines with the
\(z=-2/3\) residue as \(\mathcal C_{3,\alpha}\), while the residue at
the origin is \(\frac13\Tr(\xi(f^{p,k}))\).  The chosen contours do not
cross poles produced by nontrivial zeros of denominator zeta functions.
These residue identities give Theorem~\ref{lthm: main theorem}.

\subsection{Organization and further directions}

Section~2 fixes the test function, measures, contributing conjugacy
classes, and discriminant conventions.  Section~3 identifies the residual
spectrum and computes the traces that later occur as residues.  Section~4
rewrites the finite orbital integrals in terms of cubic orders and proves
the zeta-function functional equation.  Section~5 establishes the
approximate functional equation, and Section~6 treats the archimedean
weights.  Section~7 completes the coefficient lattice and applies Poisson
summation.  Section~8 proves the factorization and periodicity properties
of the Kloosterman sums, while the uniform evaluation of their local
factors is supplied by \cite{DengEspinosaYunZeta}.  Section~9 performs the
contour shifts and isolates the traces.

The remaining nonzero frequencies, completion term, and shifted contours
are explicit but still require estimates before they can be inserted into
the limiting averages of Beyond Endoscopy.  Understanding their combined
cancellation is the next analytic problem.

\subsection*{Acknowledgements}
\addtocontents{toc}{\protect\setcounter{tocdepth}{1}}
T. Deng is supported by Beijing Natural Science Foundation, No. 1244042 and National Natural Science Foundation of
China, No. 12401013.
M. Espinosa thanks BIMSA for its hospitality in the preparation of this work and YMSC for its unwavering support in the long process of hiring him. He also thanks James Arthur and Bill Casselman for their constant support and encouragement.

\section{The basic setup}

We  begin our work by making choices that will be constant throughout the work and drawing consequences of these choices.

\subsection{The test function}

For the rest of the paper we fix a prime $p$ and an integer $k\ge 0$. We shall reserve the letter $q$ to denote a general finite prime that, unless otherwise stated, can be equal to $p$.

For a finite prime \(q\) and integer \(r\geq0\), define
\[
f_q^{(r)}
:=\mathbf1_{\{X\in\Mat_{3\times3}(\Z_q):\,|\det X|_q=q^{-r}\}}.
\]
In particular, $f_q^{(0)}$ is the indicator of the maximal compact subgroup of $\GL(3, \Q_q)$.

\begin{definition}
    \label{def: test function}
We define the (factorizable) test function
    \begin{equation*}
        f^{p,k}=f_\infty\cdot f_p^{p,k}\cdot
        \prod_{q\neq p} f_q^{p,k},
    \end{equation*}
where
\begin{equation*}
    f_q^{p,k}=
\begin{cases}
f_q^{(0)}, & q\neq p,\\
p^{-k}\,f_p^{(k)}, & q=p,
\end{cases}
\end{equation*}
and \(f_\infty\in C_c^\infty(Z_+\backslash G(\R))\).
\end{definition}
\begin{remark}
 These are the functions that should be studied to carry out Beyond Endoscopy; see \cite[\S 2]{ARTHUR2018425}.
\end{remark}

\subsection{Regular elliptic elements, the discriminant and number theory}

We recall that for $\GL(3)$ we have the following definition.

\begin{definition}
    \label{def: regular elliptic}
    Let $\gamma\in\GL(3, \Q)$. We say that $\gamma$ is a regular elliptic element if its characteristic polynomial
    \begin{equation*}
        X^3 - aX^2 + bX - c,
    \end{equation*}
    is irreducible over $\mathbb{Q}$. We say it is regular if its characteristic polynomial has no repeated roots over $\mathbb{Q}$. 
\end{definition}
\begin{remark}
    The definition of regular elliptic given above works for $\GL(n)$ over any number field. For more general groups one must use more general definitions. 
\end{remark}

Associated to a cubic polynomial $X^3 - aX^2 + bX - c$ we define its discriminant. Concretely, we have

\begin{definition}
    Let $X^3 - aX^2 + bX - c$ be a cubic polynomial with roots $\lambda_1, \lambda_2, \lambda_3$,  then its discriminant is defined as 
    \begin{equation*}
        ((\lambda_1 - \lambda_2)(\lambda_2 - \lambda_3)(\lambda_3 - \lambda_1))^2.
    \end{equation*}
\end{definition}

The following are well known properties of the discriminant for the cubic case.

\begin{proposition}
    Let $X^3 - aX^2 + bX - c$ be a cubic polynomial. Then
    \begin{enumerate}
        \item The discriminant is a polynomial in $a, b, c$. Concretely,
            \begin{equation*}
                \Pol(a, b, c) := a^2b^2-4b^3-4a^3c+18abc-27c^2.
            \end{equation*}
        \item For $r > 0$, the discriminant satisfies
        \begin{equation*}
               \Pol(ra, r^2b, r^3c) = r^6\Pol(a, b, c).
        \end{equation*}
        
        \item The polynomial has repeated roots if and only if $\Pol(a, b, c) = 0.$

        \item The polynomial has three different real roots if and only if $\Pol(a, b, c) > 0$. Consequently, it has exactly one real root if and only if $\Pol(a, b, c) < 0$.
    \end{enumerate}
\end{proposition}

Great part of the manipulations that we will do have certain number-theoretical source associated to the following data.

\begin{definition}
Let $\gamma$ be a regular elliptic element with characteristic polynomial $X^3 - aX^2 + bX - c$. We define 
\begin{equation*}
    E_{\gamma} := \dfrac{\Q[X]}{(X^3 - aX^2 + bX - c)},
\end{equation*}
the cubic number field generated by $\gamma$. We identify $\gamma$ with the root $(X)$. 

Given a number field $E$, we denote by $D_E = D_{E/\Q}$ the discriminant of this extension, by $h_{E}$ its class number, $R_{E}$ its regulator, and $\omega_{E}$ its number of roots of unity. If this number field is produced by the characteristic polynomial of a regular elliptic element $\gamma$ we write instead $D_{\gamma}$, $h_{\gamma}$, $R_{\gamma}$, and $\omega_\gamma$.

Finally, the ring of integers of the number field will be denoted by $\cO_E$. We instead use $\cO_{\gamma}$ if this number field is produced by the characteristic polynomial of a regular elliptic element $\gamma$.  
\end{definition}

\subsection{Measures and orbital integrals}

We now explain our choices of measures. We begin with measures that do not
involve a choice of regular elliptic element. At a non-archimedean place
\(q\), we take the Haar measure on \(\GL(3,\Q_q)\) giving
\(\GL(3,\Z_q)\) volume \(1\). At the real place we use
\[
 dg_\infty=|\det g|^{-3}\prod_{1\le i,j\le3}dg_{ij}.
\]
Identify \(Z_+=\{\rho I_3:\rho>0\}\) with \(\R_{>0}\) by
\(\rho I_3\mapsto|\det(\rho I_3)|=\rho^3\), and give \(Z_+\) the
pullback of \(du/u\). Thus, in the scalar coordinate \(\rho\), its
measure is \(3\,d\rho/\rho\). All quotient measures below are induced
by these choices, and \(\GL(3,\A)\) carries the associated restricted
product measure.

We now explain measures that involve a regular elliptic element $\gamma\in \GL(3, \Q)$. Denote by $T_\gamma$ the centralizer of $\gamma$ in $\GL(3)$. We normalize the Haar measure 
\begin{equation*}
    d^\times x=\prod_v d^\times x_v
\end{equation*}
on $T_\gamma(\A)$ as follows:
\begin{itemize}
\item for each finite place $q$, $d^\times x_q$ is normalized by
$\vol\bigl(T_\gamma(\Z_q)\bigr)=1$, where
$ T_\gamma(\Z_q)=\CO_{E,q}^\times$;
\item for each real place, $d^\times x=dx/|x|$ on $\R^\times$;
\item for each complex place, $d^\times z = 2\,dx\,dy/|z|^2$ on $\C^\times$.
\end{itemize}

The above choices define the (finite) local orbital integral
\begin{equation*}
    \cO(\gamma, f_q^{p,k}) = \displaystyle\int_{T_{\gamma}(\mathbb{Q}_q)\backslash \GL(3, \mathbb{Q}_q)} f_q^{p,k}(x_q^{-1}\gamma x_q) dx_q,
\end{equation*}
where $dx_q$ is the quotient measure induced by our choices of measures above. Notice that $\gamma$ is regular elliptic over $\mathbb{Q}$, but it might not be over $\mathbb{Q}_q$. For us the following result will be fundamental.

\begin{proposition}
    \label{prop: orbital integral values}
    Let $\gamma$ be a regular semi-simple element in $\GL(3, \Q_q)$, $q$ be a prime and $r \ge 0$ be an integer. If $|\det(\gamma)|_q \neq q^{-r}$, then 
    \begin{equation*}
        \cO(\gamma, f_q^{(r)}) = 0.
    \end{equation*} 
    Consequently, if $|\det(\gamma)|_q \neq q^{-k_q}$ then $\cO(\gamma, f_q^{p,k}) = 0$, where 
    \[
   k_q:= \begin{cases}
       0, &\text{\, if \, } q\neq p\\
       k, & \text{ \, if \, }q=p.
   \end{cases}
    \]
\end{proposition}
\begin{proof}
    If $\cO(\gamma, f_q^{(r)}) \neq 0$, then a conjugate of $\gamma$ must be in the support of $f_q^{(r)}$. Since the determinant is invariant under conjugation, the result follows.
    Furthermore, because the test function $f_q^{p,k}$ at the prime $q$ is a constant multiple of $f_q^{(r_q)}$ with $r_q=0$ for $q\neq p$ and $r_p=k$, the same conclusion holds.  
\end{proof}

As we will see in later sections, we will find a way to avoid evaluating the orbital integrals explicitly. However, there is one case where we do need to know its value.

\begin{proposition}
    \label{prop: value orbital integral semisimple}
Let $q$ be a prime and $r\ge 0$ an integer. Let $\gamma$ be a regular elliptic element in $\GL(3, \Q)$ such that $|\det(\gamma)|_q = q^{-r}$. Suppose that $\gamma$ splits entirely in $\Q_q$ with roots $\lambda_1, \lambda_2$ and $\lambda_3$ in $\mathbb{Z}_q$. Then
\begin{equation*}
    \cO(\gamma, f_q^{(r)}) = |\lambda_1-\lambda_2|_q^{-1}|\lambda_2-\lambda_3|_q^{-1}|\lambda_3-\lambda_1|_q^{-1}.
\end{equation*}
\end{proposition}
\begin{proof}
   This is obtained by applying 
\cite[Corollary 4.10]{ZYun}.
\end{proof}

\subsection{Contributing Classes and their properties}

For a factorizable function, as is our test function, the global orbital integral that appears in the trace formula is
\begin{equation*}
    \cO(\gamma, f^{p,k}) =\cO(\gamma, f_{\infty}) \cdot \prod_q \cO(\gamma, f_q^{p,k}).
\end{equation*}
Proposition \ref{prop: orbital integral values} above implies conditions on the regular elliptic classes for which the above orbital integral doesn't vanish.

\begin{proposition}
\label{prop: contributing classes}
    Let $\gamma \in \GL(3, \Q)$ be a regular elliptic matrix. If 
    \begin{equation*}
        \cO(\gamma, f^{p,k}) \neq 0,
    \end{equation*}
    then $\det(\gamma) = \pm p^k$. Furthermore, in this case the characteristic polynomial of $\gamma$ is in $\mathbb{Z}[X]$. 
\end{proposition}
\begin{proof}
    Let $q$ be a prime. Given that the global orbital integral does not vanish, we know $\cO(\gamma, f_q^{p,k})\neq 0$. This implies that a conjugate of $\gamma$ lands in the support of $f_q^{p,k}$, where the local determinant exponent is \(0\) if $q\neq p$ and \(k\) if $q=p$. By definition, this support lies inside \(\Mat(3,\Z_q)\).

    The characteristic polynomial of the matrices in $\Mat(3, \Z_q)$ have coefficients in $\mathbb{Z}_q$. Furthermore, characteristic polynomial is invariant under conjugation. Consequently, the coefficients of the characteristic polynomial of $\gamma$ are in $\mathbb{Z}_q$. Because this holds for every finite prime $q$, we conclude the coefficients of the characteristic polynomials are in $\mathbb{Z}$.

    Finally, Proposition \ref{prop: orbital integral values} specifies that
    \begin{equation*}
        |\det(\gamma)|_q = 
        \begin{cases}
            1 & q\neq p,\\
        p^{-k} & q = p.\\
        \end{cases}
    \end{equation*}
    This implies that $\det(\gamma) = \pm p^k$. This concludes the proof.
\end{proof}

The above proposition motivated our following definition.

\begin{definition}
    \label{def: contributing class}
    Let $\gamma$ be a regular elliptic element in $\GL(3, \Q)$ with characteristic polynomial $X^3 - aX^2 + bX - c$. We say that $\gamma$ is a contributing conjugacy class if $\det(\gamma) = c = \pm p^k$ and $a, b\in \mathbb{Z}$. 
\end{definition}

The above discussion guarantees that the nonzero terms of the regular elliptic part of the trace formula, for our test function, are within those indexed by classes whose characteristic polynomials are of the form
\begin{equation*}
    X^3 - aX^2 + bX \mp p^k, \qquad a, b\in\mathbb{Z}.
\end{equation*}
We now set up the following convention.

\begin{notation}
    Let $\gamma$ be a contributing class with characteristic polynomial $X^3 - aX^2 + bX \mp p^k$. We define
    \begin{equation*}
        \Pol(a, b) := \Pol(a, b, \pm p^k). 
    \end{equation*}
\end{notation}
\begin{remark}
    A contributing class is implicitly associated to a choice of sign $\pm$. We shall suppress that dependence in some of the notation. However, any time a construction that a priori depends on all the coefficients $a, b, c$ of the characteristic polynomial, appears as a function of $a, b$ alone, is because we have fixed $c$ as $\pm p^k$.
\end{remark}
 
\section{The Residual Representations and their Traces}

\subsection{The residual representations that will appear}
For the sake of clarity let us begin by elaborating on which traces we are looking to eliminate via this process. 

\begin{proposition}
For the test function of Definition~\ref{def: test function}, the only
residual representation with nonzero trace is the trivial representation.
Consequently, the residual contribution to the trace is
\(\Tr(\mathbf 1(f^{p,k}))\).
\end{proposition}
\begin{proof}
The discrete spectrum of \(\GL(n)\) decomposes as
\begin{equation*}
    L^2(Z^+\GL(n, \mathbb{Q})\backslash \GL(n, \mathbb{A}_{\mathbb{Q}})) = \bigoplus_{d\mid n} \hat{\bigoplus_{\sigma}} \mbox{Speh}(\sigma, n/d),
\end{equation*}
where \(\sigma\) runs over the cuspidal spectrum of \(\GL(d,\Q)\);
see \cite{MWSpecRes,GetzAut}.  For \(n=3\), the residual summands are
\(\operatorname{Speh}(\sigma,3)\), with \(\sigma\) a Hecke character,
and
\begin{equation*}
\operatorname{Speh}(\sigma,3)
=\sigma|\cdot|\boxplus\sigma\boxplus\sigma|\cdot|^{-1}.
\end{equation*}
Every finite component \(f_q^{p,k}\) is bi-\(\GL(3,\Z_q)\)-invariant.
Hence it annihilates a representation with no spherical vector.  A
residual summand that contributes must therefore have \(\sigma\)
unramified at every finite place.  Since the central character
\(\sigma^3\) must also be trivial on \(Z_+\), the class-number-one
description of unramified Hecke characters of \(\Q\) forces
\(\sigma=\mathbf1\).  Finally,
\(\operatorname{Speh}(\mathbf1,3)=\mathbf1\), which proves the claim.
\end{proof}

\subsection{The trace of the trivial representation at the finite places}

Now that we know the representation whose contribution we must eliminate, let us find this contribution.

\begin{proposition}
With the preceding normalizations,
\[
 \Tr(\mathbf1(f^{p,k}))
 =p^k\frac{1-p^{-(k+1)}}{1-p^{-1}}
       \frac{1-p^{-(k+2)}}{1-p^{-2}}
  \int_{Z_+\backslash\GL(3,\R)}f_\infty(g)\,dg.
\]
\end{proposition}
\begin{proof}
Let \(\pi=\mathbf1\) be the one-dimensional trivial representation.
By definition, the convolution operator attached to \(f^{p,k}\) acts on
\(v\in\C\) by
\[
\pi(f^{p,k})v
=\int_{Z_+\backslash\GL(3,\A)}f^{p,k}(g)\pi(g)v\,dg
=\left(\int_{Z_+\backslash\GL(3,\A)}f^{p,k}(g)\,dg\right)v.
\]
 Thus 
\[
\Tr(\mathbf1(f^{p,k}))
=\left(\int_{Z_+\backslash\GL(3,\R)}
f_\infty(g)\,dg\right)
\prod_{q<\infty}
\left(\int_{\GL(3,\Q_q)}f_q^{p,k}(g_q)\,dg_q\right).
\]
It remains only to compute the finite local integrals.  For \(q\ne p\),
\(f_q^{p,k}=\mathbf1_{\GL(3,\Z_q)}\), whose integral is \(1\) under our
normalization.  Put
\[
 K_p=\GL(3,\Z_p),\qquad
 M_{p,k}=\{g\in\Mat_3(\Z_p):v_p(\det g)=k\}.
\]
The map \(gK_p\mapsto g\Z_p^3\) identifies the right \(K_p\)-cosets
in \(M_{p,k}\) with the sublattices of \(\Z_p^3\) of index \(p^k\).
Since \(\vol(K_p)=1\), \(\vol(M_{p,k})\) is the number of these
lattices.  Hermite normal form gives
\[
\sum_{k\ge0}\vol(M_{p,k})X^k
=\sum_{a_1,a_2,a_3\ge0}
  p^{2a_1+a_2}X^{a_1+a_2+a_3}
=\frac1{(1-X)(1-pX)(1-p^2X)}.
\]
Hence
\[
\vol(M_{p,k})
=\frac{(p^{k+1}-1)(p^{k+2}-1)}{(p-1)(p^2-1)}
=p^{2k}\frac{1-p^{-(k+1)}}{1-p^{-1}}
          \frac{1-p^{-(k+2)}}{1-p^{-2}}.
\]
Since \(f_p^{p,k}=p^{-k}\mathbf1_{M_{p,k}}\), its integral is
\[
p^k\frac{1-p^{-(k+1)}}{1-p^{-1}}
    \frac{1-p^{-(k+2)}}{1-p^{-2}}.
\]
Substitution of this \(p\)-adic integral into the factored trace formula
above proves the proposition.
\end{proof}

\subsection{The trace of the trivial representation at the infinite place}
It will be important for us to rewrite the Archimedean orbital integral as well. We recall that given a regular element $\gamma$ the Archimedean part of the trace integral is
\begin{equation*}
    \cO(\gamma, f_{\infty})
    =\int_{T_\gamma(\R)\backslash\GL(3,\R)}
      f_{\infty}(x^{-1}\gamma x)\,dx.
\end{equation*}
This function is invariant under conjugation, which is to say, that it only depends on the conjugacy class of $\gamma$. For purposes of our manipulation, we shall require the Weyl Integration formula for $Z_+\backslash \GL(3, \R)$. We now make the required preparations.

\begin{proposition}
\begin{enumerate}
    \item The group $\GL(3, \R)$ has two families of maximal torus: the split family torus $T_s$ and a family of mixed torus $T_m$. The corresponding Weyl Groups cardinalities are $|W(T_s, \R)| = 6$ and $|W(T_m, \R)| = 2$. 

    \item The group $\GL(3, \R)$ has two families of maximal torus: the split family torus $Z_+\backslash T_s$ and a family of mixed torus $Z_+\backslash T_m$. The corresponding Weyl Groups cardinalities are $|W(Z_+\backslash T_s, \R)| = 6$ and $|W(Z_+\backslash T_m, \R)| = 2$. 
\end{enumerate}
\end{proposition}
\begin{proof}
Part (a) follows from definition and by noticing that in the split torus there are 6 ways to permute the roots, while on the mixed torus only the complex pair can be permuted. Part (b) follows from part (a) because $Z_+$ is a subgroup of every maximal torus.    
\end{proof}

In our setup, we know that $f_{\infty}\in C_c(Z_+\backslash \GL(3, \R))$ and that its orbital integrals have compact support. Associated to this function and to each of the maximal tori families there is a corresponding orbital integral. Concretely, for a maximal torus $T$ of $\GL(3, \R)$, we define
\begin{align*}
    \Phi_T(t) = \displaystyle\int_{T\backslash \GL(3, \R)} f_{\infty}(xtx^{-1})dx.
\end{align*}
Write
\[
D_G(t):=\left|\det\!\left(1-\operatorname{Ad}(t);
  \mathfrak{gl}_3/\mathfrak t\right)\right|
\]
for the Weyl discriminant.
Using this notation, the Weyl integration formula for $Z_+\backslash \GL(3, \R)$
\begin{equation*}
    \displaystyle\int_{Z_+\backslash \GL(3, \R)} f_\infty(g)dg = \dfrac{1}{6}\displaystyle\int_{Z_+\backslash T_s}\Phi_{T_s}(t) D_G(t)\,dt + \dfrac{1}{2}\displaystyle\int_{Z_+\backslash T_m}\Phi_{T_m}(t) D_G(t)\,dt.
\end{equation*}
We are supposing that the measures here are normalized as usual to satisfy the Weyl Integration formula and the integration by stages formula. At this stage, each of the involved integrals is the coefficient on the torus and not in the space of coefficients of the characteristic polynomial. We now change the integration space. 

We need concrete representatives of our two families of tori. We specify them through their eigenvalues. We set
\begin{align*}
    T_s &= \{ \diag(\lambda_1, \lambda_2, \lambda_3) \mid \lambda_i\neq 0 \},\\
    T_m &= \{ \diag(\lambda,  re^{i\theta}, re^{-i\theta}) \mid \lambda \neq 0, r>0 \}
\end{align*}

The space of coefficients is given by
\begin{equation*}
    \{(c_1, c_2, c_3)\in\R^3 \mid\, X^3 -c_1X^2 + c_2X - c_3 \mbox{ has no repeated roots}\} \subset\R^3.
\end{equation*}
This is an open set given as the complement of $\Pol(c_1, c_2, c_3)=0$ and with $c_3\neq 0$.
The coefficient maps are given as follows. For the split torus $T_s$ we have
\begin{align*}
    c_1 &= \lambda_1 + \lambda_2 + \lambda_3,\\
    c_2 &= \lambda_1\lambda_2 + \lambda_2\lambda_3 + \lambda_3\lambda_1, \\
    c_3 &= \lambda_1\lambda_2\lambda_3.
\end{align*}
while, for the mixed torus $T_m$, we have
\begin{align*}
    c_1 &= \lambda + 2r\cos(\theta),\\
    c_2 &= 2\lambda r\cos(\theta) + r^2, \\
    c_3 &= \lambda r^2.
\end{align*}
The Weyl integration formula assigns the measure \(D_G(t)\,dt\) to
each torus. In coefficient space let \(dc=dc_1dc_2dc_3\).

\begin{proposition}
\label{prop: measure change}
For each maximal torus \(T\), the coefficient map has degree
\(|W(T)|\), and
\[
(t\mapsto(c_1,c_2,c_3))_*\bigl(D_G(t)\,dt\bigr)
=|W(T)|\,|\Pol(c_1,c_2,c_3)|^{1/2}|c_3|^{-3}\,dc.
\]
\end{proposition}
\begin{proof}
The Jacobian of the root-to-coefficient map is the Vandermonde
\(|\Pol(c_1,c_2,c_3)|^{1/2}\).  Expressing Haar measure on
\(\GL(3,\R)\) as \(|\det|^{-3}\) times Lebesgue measure supplies the
factor \(|c_3|^{-3}\), and the degree of the map supplies
\(|W(T)|\).
\end{proof}

The computations above do not take into account the invariance under the center. To now take this into account, we must restrict on both parts correspondingly. In order to single out a unique representative under the action of $Z_+$ we impose the condition that the determinant has absolute value $1$. That is, determinant will be $1$ or $-1$. In other words, in coefficient space, we shall now work in 
\begin{equation*}
    \{(c_1, c_2)\in\R^2 \mid\, X^3 -c_1X^2 + c_2X \mp 1 \mbox{ has no repeated roots}\} \subset\R^2.
\end{equation*}
We are identifying this set with a subset of the hyperplanes $(c_1, c_2, \pm 1)$ which are parallel to the coordinate planes. Thus the Lebesgue measure in them is still $dc_1dc_2$. 

Similarly, for $T_s$ and $T_m$ the determinant condition implies, respectively,
\[
 \lambda_1\lambda_2\lambda_3=\pm1,
 \qquad
 \lambda r^2=\pm1.
\]
Disintegrating the preceding measure along the central variable, rather
than restricting a three-dimensional measure to a measure-zero slice,
gives the quotient formula.
\begin{proposition}
\label{prop: change of variable in coefficient space}
For each maximal torus \(T\), on the determinant-\(\epsilon\) slice
of \(Z_+\backslash T\), where \(\epsilon\in\{\pm1\}\),
\[
(t\mapsto(c_1,c_2))_*\bigl(D_G(t)\,d\dot t\bigr)
=|W(T)|\,|\Pol(c_1,c_2,\epsilon)|^{1/2}\,dc_1\,dc_2.
\]
\end{proposition}
\begin{proof}
Write \(c_3=\epsilon\rho^3\), \(c_1=\rho b_1\), and
\(c_2=\rho^2b_2\).  Disintegrate the formula of
Proposition~\ref{prop: measure change} against the
determinant-normalized central measure \(d(\rho^3)/\rho^3\).
The remaining measure on the slice is the one displayed above.
\end{proof}

Given all our discussion, we can now give the following definition.

\begin{definition}
\label{def:orbital integral in coefficient space}
    Fix a sign $\pm$ and define the open dense set
    \begin{equation*}
    D^{\pm} := \{(c_1, c_2)\in\R^2 \mid\, X^3 -c_1X^2 + c_2X \mp 1 \mbox{ has no repeated roots}\} \subset\R^2.
    \end{equation*}
    Define the function $\theta^{\pm}: D^{\pm} \longrightarrow \mathbb{C}$ as
    \begin{equation*}
        \theta^{\pm}(x, y) = |\Pol(x, y, \pm 1)|^{1/2}\displaystyle\int_{T(x, y)\backslash \GL(3, \R)} f_{\infty}(g\cdot t(x,y)\cdot g^{-1})dg,
    \end{equation*}
    where $t(x, y)$ is an element of a corresponding maximal torus $T(x, y)$ with parameters $(x, y)$. Here $dg$ is the corresponding measure given by the Weyl integration formula.
\end{definition}
\begin{remark}
   These are \textit{normalized }orbital integrals, which we know are smooth in the regular set. By our choice of test function, they are of compact support and thus integrable in $\R^2$. However, as they approach the singular set, the derivatives have jump singularities.
\end{remark}

Our conclusion is the following.

\begin{proposition}
    In the context of the above discussion,
    \begin{equation*}
        \displaystyle\int_{Z_+\backslash GL(3, \mathbb{R})} f_{\infty}(x) dx = \displaystyle\sum_{\pm}\displaystyle\int_{\R^2} \theta^{\pm}(x, y) \, dxdy.
    \end{equation*}
\end{proposition}
\begin{proof}
    The Weyl integration formula gives
    \begin{equation*}
    \displaystyle\int_{Z_+\backslash \GL(3, \R)} f_\infty(g)dg = \dfrac{1}{6}\displaystyle\int_{Z_+\backslash T_s}\Phi_{T_s}(t) D_G(t)\,dt + \dfrac{1}{2}\displaystyle\int_{Z_+\backslash T_m}\Phi_{T_m}(t) D_G(t)\,dt.
\end{equation*}
Substituting the change of variables of Proposition \ref{prop: change of variable in coefficient space} this becomes
\begin{align*}
    \, 
    &\, \displaystyle\sum_{\pm}\displaystyle\int_{\Pol(c_1, c_2, \pm 1)>0}|\Pol(c_1, c_2, \pm 1)|^{1/2}\Phi_{T_s}(t) dc  
    + \displaystyle\sum_{\pm}\displaystyle\int_{\Pol(c_1, c_2, \pm 1)<0}|\Pol(c_1, c_2, \pm 1)|^{1/2}\Phi_{T_m}(t) dc\\
    &= \displaystyle\sum_{\pm}\displaystyle\int_{\Pol(c_1, c_2, \pm 1)>0}\theta^{\pm}(c_1, c_2) dc  
    + \displaystyle\sum_{\pm}\displaystyle\int_{\Pol(c_1, c_2, \pm 1)<0}\theta^{\pm}(c_1, c_2) dc\\
    &= \displaystyle\sum_{\pm}\displaystyle\int_{\mathbb{R}^2}\theta^{\pm}(c_1, c_2) dc.
\end{align*}
\end{proof}
Consequently, we have the following.
\begin{proposition}
\label{prop: value trivial}
    In the context of the setup established in the previous section, we have
    \begin{equation*}
        \Tr(\mathbf{1}(f^{p,k})) = p^{k}\cdot\dfrac{1- p^{-(k + 1)}}{1 - p^{-1}}\cdot \dfrac{1- p^{-(k + 2)}}{1 - p^{-2}} \cdot \displaystyle\sum_{\pm}\displaystyle\int_{\R^2} \theta^{\pm}(x, y) \, dxdy.
    \end{equation*}
\end{proposition}

\subsection{The trace from the induced representation}

As we have established, the trivial representation is obtained as the isobaric sum from the trivial representation from the  split torus. Consequently, we also expect to find a contribution from the normalized induced representation itself. We shall call this induced representation $\xi$. The space where it acts is the set of functions $\psi: \GL(3, \A) \longrightarrow \C$ invariant under $Z_+$ and such that
\begin{equation*}
    \psi(bg) = \abs{\dfrac{x}{z}}\psi(g),
\end{equation*}
where
\begin{equation*}
    b = 
    \begin{pmatrix}
        x & * & *\\
        0 & y & *\\
        0 & 0 & z
    \end{pmatrix}
\end{equation*}
is any element of the Borel subgroup.

\begin{proposition}
\label{prop: value of Eisenstein term}
For the fixed test function \(f^{p,k}\),
\[
\Tr(\xi(f^{p,k}))
=3(k+1)(k+2)\sum_{\epsilon=\pm1}
 \int_{\Pol(c_1,c_2,\epsilon)>0}
 \theta^\epsilon(c_1,c_2)
 |\Pol(c_1,c_2,\epsilon)|^{-1/2}\,dc_1\,dc_2.
\]
\end{proposition}
\begin{proof}
Put \(K_q=\GL(3,\Z_q)\). Each finite function \(f_q^{p,k}\) is
bi-\(K_q\)-invariant, so \(\xi_q(f_q^{p,k})\) has image in the
one-dimensional spherical line \(\xi_q^{K_q}\). Consequently,
\[
\Tr(\xi(f^{p,k}))
=\Tr(\xi_\infty(f_\infty))
 \prod_{q<\infty}\Tr(\xi_q(f_q^{p,k})).
\]
For \(q\ne p\), \(f_q^{p,k}=\mathbf1_{K_q}\) is the projection onto
\(\xi_q^{K_q}\), and its trace is \(1\). Under the normalized Satake
isomorphism (see \cite{Casselman})
\[
\mathcal S\!\left(p^{-k}f_p^{(k)}\right)
=h_k(X_1,X_2,X_3).
\]
Here \(h_k\) denotes the complete homogeneous symmetric polynomial of
degree \(k\):
\[
h_k(X_1,X_2,X_3)
=\sum_{\substack{i_1,i_2,i_3\geq0\\i_1+i_2+i_3=k}}
X_1^{i_1}X_2^{i_2}X_3^{i_3}.
\]
The Satake parameters of
\(\xi_p=\operatorname{Ind}_{B(\Q_p)}^{\GL(3,\Q_p)}(\mathbf1)\) are
\((1,1,1)\). Therefore
\[
\Tr(\xi_p(f_p^{p,k}))=h_k(1,1,1)
=\binom{k+2}{2}=\frac{(k+1)(k+2)}2.
\]

We next compute the archimedean trace. Let \(A=T_s\) be the split
diagonal torus and \(B=AN\) the upper-triangular Borel. For regular
\(t\in A\), apply
\cite[Proposition~8.1, equation~(8.2)]{Adams2011} with
\(P=B\), \(M=A\), and \(\sigma=\mathbf1\). Since \(A\) has no roots
in \(M=A\), the factor \(D(\Delta_i^+,wt)\) in that formula is \(1\).
Moreover,
\[
W(M(\R),A(\R))=\{1\},\qquad W(G(\R),A(\R))=S_3,
\]
and \(\Theta_\sigma(wt)=1\) for every \(w\in S_3\). Thus
\[
\Theta_{\xi_\infty}(t)
=\frac{1}{|D(\Delta^+,t)|}\sum_{w\in S_3}1
=\frac{6}{|D(\Delta^+,t)|}.
\]
By \cite[Definition~2.4]{Adams2011},
\[
|D(\Delta^+,t)|^2
=\left|\det\!\left(1-\operatorname{Ad}(t);
       \mathfrak{gl}_3/\mathfrak a\right)\right|
=D_G(t).
\]
Hence
\[
D_G(t)^{1/2}\Theta_{\xi_\infty}(t)=6
\qquad(t\in A_{\mathrm{reg}}).
\]
The same formula sums over the real conjugates of \(t\) that lie in the
inducing torus \(A\). A regular element of the mixed torus has one real
eigenvalue and one nonreal conjugate pair, so the sum is empty and
\(\Theta_{\xi_\infty}\) vanishes there.

Apply the Weyl integration formula to
\(f_\infty\Theta_{\xi_\infty}\). The mixed-torus term vanishes and the
factor \(6\) above cancels \(|W(A,\R)|=6\), giving
\[
\Tr(\xi_\infty(f_\infty))
=\int_{Z_+\backslash A}
  \cO(t,f_\infty)D_G(t)^{1/2}\,d\dot t.
\]
Every class in \(Z_+\backslash A\) has a unique representative with
\(\det t=\epsilon\), where \(\epsilon\in\{\pm1\}\). We split the last
integral into these two determinant slices. On the \(\epsilon\)-slice,
if the characteristic polynomial is
\(X^3-c_1X^2+c_2X-\epsilon\), then
\[
D_G(t)=|\Pol(c_1,c_2,\epsilon)|.
\]
Proposition~\ref{prop: change of variable in coefficient space} and
Definition~\ref{def:orbital integral in coefficient space} now give
\begin{align*}
\Tr(\xi_\infty(f_\infty))
&=6\sum_{\epsilon=\pm1}
  \int_{\Pol(c_1,c_2,\epsilon)>0}
  \cO(t(c_1,c_2),f_\infty)\,dc_1\,dc_2\\
&=6\sum_{\epsilon=\pm1}
  \int_{\Pol(c_1,c_2,\epsilon)>0}
  \theta^\epsilon(c_1,c_2)
  |\Pol(c_1,c_2,\epsilon)|^{-1/2}\,dc_1\,dc_2.
\end{align*}
Multiplying by \(\binom{k+2}{2}\) proves the formula.
\end{proof}

\section{The manipulation of the regular elliptic part:
\texorpdfstring{\(p\)-adic}{p-adic} Orbital Integrals}

\subsection{The singular set and the missing lattice points}

The regular elliptic part of the trace formula for $\GL(3, \Q)$ is defined as
\begin{equation*}
    \displaystyle\sum_{[\gamma]}\vol(\gamma)\cO(\gamma, f^{p,k}),
\end{equation*}
where the index runs over conjugacy classes of $\GL(3, \mathbb{Q})$ which are regular elliptic. As we have seen in Proposition \ref{prop: contributing classes}, the characteristic polynomials that parametrize a conjugacy class that contribute to the sum are of the form
\begin{equation*}
    X^3 - aX^2 + bX \mp p^k,
\end{equation*}
as long as this polynomial is irreducible over $\mathbb{Q}$. 

\begin{notation}
    Whenever $X^3 - aX^2 + bX \mp p^k$ is irreducible over $\mathbb{Q}$, we shall denote by $\gamma(a, b)$ the unique regular elliptic conjugacy class with this characteristic polynomial. 
\end{notation}

The above discussion allows us to rewrite the sum over the conjugacy classes as a sum over the pairs $(a, b)$ and signs $\pm$ that produce regular elliptic elements. That is,
\begin{equation*}
    \displaystyle\sum_{[\gamma]}\vol(\gamma)\cO(\gamma, f^{p,k}) = \displaystyle\sum_{\pm}\displaystyle\sum_{(a, b)\in V(\pm)} \vol(\gamma(a, b))\cO(\gamma(a, b), f^{p,k}),
\end{equation*}
where $V(\pm)$ is the set of integral $(a, b)$ such that $X^3 - aX^2 + bX \mp p^k$ is regular elliptic. When there is no risk of confusion we shall simply write $(a, b)\in V$. We shall also write $\gamma$ instead of $\gamma(a, b)$ unless we require to emphasize the parameters.

\begin{remark}
\label{rmk: set of singular lines}
For $\GL(3)$ there is no criterion for irreducibility of a polynomial over $\mathbb{Q}$ associated to the value of the discriminant $Pol(a, b, c)$. In particular, contrary to what happens in $\GL(2)$, $Pol(a, b, c)\neq 0$ implies regularity but not ellipticity. A criterion for irreducibility, however, is given by the rational root test: $X^3 - aX^2 + bX \mp p^k$ is irreducible over $\mathbb{Q}$ if and only if $\pm p^t$ is not a root for $t=0, ..., k$. These are a finite number of linear equations in $(a, b)$. In particular, this implies
\begin{align*}
    V(+) &= \{(a, b)\in\mathbb{Z}^2 \mid  (\pm p)^{3t} - a(\pm p)^{2t} + b (\pm p) - p^k \neq 0, t = 0, 1, ..., k\},\\
    V(-) &= \{(a, b)\in\mathbb{Z}^2 \mid  (\pm p)^{3t} - a(\pm p)^{2t} + b (\pm p) + p^k \neq 0, t = 0, 1, ..., k\}.
\end{align*}
\end{remark}

\subsection{The volume term}

Let $E$ be a number field. Consider  
\[
\lambda:\; \A^\times_E\longrightarrow \R_{>0},\qquad  x\mapsto \prod_v|\det(x)|_v.
\]

Let $\A^{\times, 1}_E=\Ker(\lambda)$ and 
$d\mu$ be the Haar measure on $\A^{\times,1}_E$ such that 
\[
dt=d\mu\otimes \lambda^*(\frac{dx}{|x|}).
\]

\begin{proposition}
\label{def: volume value}
Let $\gamma\in \GL(n, \Q)$ be regular elliptic, then
\begin{equation*}
    \vol(\gamma):=\vol(d\mu, E^\times \backslash \A^{\times,1}_E) = |D_{E/\mathbb{Q}}|^{1/2} \cdot\left.\frac{\zeta_E(s)}{\zeta_{\Q}(s)}\right\vert_{s=1},
\end{equation*}
where $E=\Q[\gamma]$ and the measure is chosen above and 
 $\zeta_E(s)$ (resp. $\zeta_{\Q}(s)$) is the zeta function of $E$ (resp. $\Q$). 
\end{proposition}
\begin{proof}
We sketch a proof of this basic fact.
Let $T_\gamma$ be the centralizer of 
$\gamma$ in $\GL(n)$ such that 
\[
T_\gamma=\Res_{E/\Q}\G_m.
\]
We want to apply \cite[Theorem 4.1]{Tate67} to conclude
our computation. But the choice Haar measure 
in \cite{Tate67} is different from ours, in fact 
by \cite[Lemma 2.3.3]{Tate67}, we have 
\[
d^\times x_p =N(\mathfrak D_v)^{1/2}d^\times\tilde x_p
\]
where $d^\times\tilde x_p$ denotes the self-dual Haar measure in \cite{Tate67} and $\mathfrak D_v$ is the local different of $E_v$. This shows in fact 
\[
d^\times x=\prod_{v}N(\mathfrak D_v)^{1/2}d^\times \tilde x=|D_{E/\mathbb{Q}}|^{1/2} d^\times \tilde x
\]
and $d\mu=|D_{E/\mathbb{Q}}|^{1/2} d\tilde \mu$, 
where $d\tilde \mu$ is the self dual measure. Finally, applying \cite[Theorem 4.1]{Tate67},
\[
\vol(d\tilde\mu, E^\times \backslash \A^{\times,1}_E)=\Res_{s=1}\zeta_E(s)=\left.\frac{\zeta_E(s)}{\zeta_{\Q}(s)}\right\vert_{s=1}
\]
which implies 
\[
 \vol(\gamma)= |D_{E/\mathbb{Q}}|^{1/2} \cdot\left.\frac{\zeta_E(s)}{\zeta_{\Q}(s)}\right\vert_{s=1},
\]
\end{proof}   
\begin{remark}
With our chosen measure normalizations one has, in terms of the class number $h_{\gamma}$,
regulator $R_{\gamma}$, and number of roots of unity $\omega_\gamma$ of the field $E = E_{\gamma}=\Q[x]/(P_\gamma(x))$ with discriminant $D_{E/\Q}$,
\begin{equation}\label{eq:vol-classical}
\vol(\gamma)=
\frac{2^{r_1}(2\pi)^{r_2}\,h_{\gamma} R_{\gamma}}
{w_{\gamma}}
\end{equation}
where $r_1$ (resp. $r_2$) is the number of real places (resp. complex places).
\end{remark}
\subsection{The  $p$-adic Orbital Integrals}
\subsubsection{The order associated to $\gamma(a, b)$}

We now begin our study of the orbital integral at the finite places. We had previously defined $E_{\gamma}$. We are now constructing one such $\gamma$ from an appropriate pair $(a, b)$, thus instead of writing $\gamma(a, b)$ as the subindex we simply write $(a, b)$. For example, the field $E_{(a, b)}$ is $E_{\gamma(a, b)}$.

\begin{definition}
\label{def: number theory of gamma}
 Let $(a, b)\in V(\pm)$. We denote the ring of integers of the extension $E(a, b)/\Q$ by $\cO_{(a, b)}$ or $\cO_E$ if there is no possibility of confusion. As the polynomial $X^3 - aX^2 + bX \mp p^k$ has integer coefficients, it generates a cubic monogenic order (via its root $(X)$) inside $E(a, b)$. We denote this order by $R(a, b)$.
\end{definition}
\begin{remark}
 Under the identification of $\gamma$ with the root $(X)$ in $E(a, b)$ this order is $\mathbb{Z}[\gamma]$.
\end{remark}

The standard theory of orders, as for example discussed in \cite{LangANT}, implies the following result.

\begin{lemma}
In the context of the above discussion, $\abs{\Pol(a, b)}$ is the discriminant of $R(a, b)$ and there exists a positive integer $s_{\gamma}$ such that
\begin{equation*}
    \abs{\Pol(a, b)} = s_{\gamma}^2 D_{E/\Q}.
\end{equation*}
\end{lemma}
\begin{proof}
 This follows from the standard calculation using basis, we omit the details.
\end{proof}

\begin{proposition}
    Let $q$ be a prime number that does not divide $s_{\gamma}$, then
    \begin{equation*}
        \cO(\gamma, f_q^{p,k}) = 1.
    \end{equation*}
    Consequently, the product of finite orbital integrals is the finite product
    \begin{equation*}
        \prod_{q\mid s_{\gamma}} \cO(\gamma, f_q^{p,k})
    \end{equation*}
\end{proposition}
\begin{proof}
Fix a prime $q$. Recall that in \cite[Theorem 2.5]{ZYun}, we have a polynomial $P(t)\in 1+t\Z[t]$ such that 
\[
\cO(\gamma, f_q^{p,k})=P(1).
\]
It follows from \cite[Theorem 2.5 (1)]{ZYun} that the degree of $P(t)$ is  $2v_q(s_\gamma)$. In particular, if $q\nmid s_\gamma$, we have $P(t)=1$, which proves the proposition.    
\end{proof}

\subsubsection{The zeta function associated to an order}
\label{sub: The zeta function associated to an order}

The first obstacle in our current case is to express the finite product of orbital integrals in a way that is amenable for our manipulations. For this purpose, we must first study certain zeta functions associated to cubic orders. We introduce the following definition.

\begin{definition}
Let $q$ be a prime. Let $R$ be a cubic order in a cubic algebra over $\Q$. Let $\{\fm_i\}_{i=1}^r$ be the maximal ideals of $R_{q}$, the completion of $R$ at $q$ and $k_i:=R_q/\fm_i$ the corresponding residue fields. We define
\begin{equation}
\label{eq:def-zetaS-local}
\zeta_{R_q}(s)\;:=\;\prod_{i=1}^r \frac{1}{1-|k_i|^{-s}},
\end{equation}
and
\begin{equation}
\label{eq:def-zetaS}
\zeta_R(s)=\prod_{q}\zeta_{R_q}(s).
\end{equation}
\end{definition}
Out of the above zeta functions we generate quotients, which will allow us to recover the quotients that come from the volume term. We define the following.
\begin{definition}
Let $q$ be a prime. Let $R$ be a cubic order in a cubic algebra over $\Q$. 
\begin{equation}
\label{eq:def-L}
L_{R_q}(s): =\frac{\zeta_{R_q}(s)}{\zeta_{\Q_q}(s)}.
\end{equation}
and
\begin{equation}
\label{eq:L}
L_{R}(s):=\prod_{q}L_{R_q}(s).
\end{equation}
\end{definition}

Finally, we introduce the following function.

\begin{definition}
Let $q$ be a prime. Given   a cubic order $R$ in a cubic algebra over $\Q$, we define 
\[
h_{R_q}(s)=\begin{cases}
1, &\text{\, if $R_q$ is Gorenstein. }\\
1+q^{1-s},  &\text{\, otherwise. }\\
\end{cases}
\]  
Globally, we define the multiplicative function 
\begin{equation*}
    h_R(s)=\prod_{q}h_{R_q}(s).
\end{equation*}
\end{definition}

The definition is justified by Proposition \ref{prop: Cohen-Macaulay-2}
and Proposition \ref{prop: Cohen-Macaulay-type}.

\begin{definition}
    Given a finite flat algebra $R$ over $\Z_q$, we define its Cohen-Macaulay type 
    at a maximal ideal $\mathfrak p$:
    \[
    \mathrm{Type}_{\mathfrak p}(R)=\dim_{ R/\mathfrak p}R^\vee/\mathfrak p R^\vee.
    \]
    and its Cohen-Macaulay type 
    \[
    \mathrm{Type}(R)=\max_{\mathfrak p}\{\mathrm{Type}_{\mathfrak p}(R)\}.
    \]
\end{definition}

The key fact for us 
\begin{proposition}
 \label{prop: Cohen-Macaulay-2}
We have 
\begin{itemize}
    \item[(1)] We have $R$ is Gorenstein if and only if $\mathrm{Type}(R)=1$.
    \item[(2)]If   $\mathrm{Type}(R)=2$, then 
   for any fractional ideal $I$ such that $\End_R(I)=R$ we have $I\cong R$ or $I\cong R^\vee$.
\end{itemize}
\end{proposition}
\begin{proof}
    Part (1) follows from \cite[Proposition 3.4]{MS24} and part (2) follows from \cite[Theorem 3]{MS24}.
\end{proof}

\begin{proposition}
\label{prop: Cohen-Macaulay-type}
Let \(R\) be a cubic ring over \(\mathbb Z_q\), i.e. a finite flat
\(\mathbb Z_q\)-algebra of rank \(3\). Then for every maximal ideal
\(\mathfrak m\subset R\), the local ring \(R_{\mathfrak m}\) is a
one-dimensional Cohen--Macaulay local ring of type at most \(2\).
Consequently, \(R\) has Cohen--Macaulay type at most \(2\).
\end{proposition}

\begin{proof}
Since \(R\) is semi-local, it is enough to prove that
\(\mathrm{type}(R_{\mathfrak m})\le 2\) for all maximal ideals
\(\mathfrak m\subset R\). So fix \(\mathfrak m\), and set
\[
A:=R_{\mathfrak m},\qquad k:=A/\mathfrak m A,\qquad B:=A/pA.
\]

Because \(R\) is free over \(\mathbb Z_q\), the localization \(A\) is
torsion-free over \(\mathbb Z_q\). Hence, multiplication by \(q\) on \(A\)
is injective. In addition, \(A\) is finite over \(\mathbb Z_q\), so \(A\)
is integral over \(\mathbb Z_q\). Therefore
\[
\dim A=\dim \mathbb Z_q=1.
\]
Thus, \(q\) is a non-zero divisor and a parameter of \(A\), so
\[
\depth A\ge 1=\dim A.
\]
Hence \(A\) is Cohen--Macaulay (cf. \cite[Proposition (I6.4.6) and Definition (I6.5.1)]{EGA-IV-I}).

We now identify the Cohen--Macaulay type of \(A\). Apply
\(\operatorname{Hom}_A(k,-)\) to the short exact sequence
\[
0\longrightarrow A \xrightarrow{\times q} A \longrightarrow B \longrightarrow 0.
\]
This gives an exact sequence
\[
0\to \operatorname{Hom}_A(k,A)\xrightarrow{\times q}\operatorname{Hom}_A(k,A)
   \to \operatorname{Hom}_A(k,B)\to \operatorname{Ext}^1_A(k,A)
   \xrightarrow{\times q}\operatorname{Ext}^1_A(k,A).
\]
Now \(\operatorname{Hom}_A(k,A)=0\)
since \(q\) is a non-zero divisor on \(A\). So, the exact
sequence simplifies to
\[
0\longrightarrow \operatorname{Hom}_A(k,B)\longrightarrow \operatorname{Ext}^1_A(k,A)
\xrightarrow{\times q}\operatorname{Ext}^1_A(k,A).
\]
But \(\operatorname{Ext}^1_A(k,A)\) is naturally a \(k\)-vector space and
is annihilated in particular by \(q\). Therefore, we obtain an isomorphism
\[
\operatorname{Ext}^1_A(k,A)\cong \operatorname{Hom}_A(k,B).
\]
Since \(\operatorname{Hom}_A(k,B)=0:_B\mathfrak mA=\operatorname{Soc}(B)\), it
follows that
\[
\mathrm{type}(A)
=\dim_k \operatorname{Ext}^1_A(k,A)
=\dim_k \operatorname{Soc}(B).
\]

So it remains to bound the \(k\)-dimension of the socle of \(B\).

Because \(R\) is free of rank \(3\) over \(\mathbb Z_q\), the quotient
\(R/qR\) is a \(3\)-dimensional \(\mathbb F_q\)-vector space. In particular,
\(R/qR\) is Artinian. Let \(\overline{\mathfrak m}\) be the image of
\(\mathfrak m\) in \(R/qR\). Then
\[
B=A/qA\cong (R/qR)_{\overline{\mathfrak m}}.
\]
For an Artinian ring, localization at a maximal ideal is a direct factor, so
\((R/qR)_{\overline{\mathfrak m}}\) is an \(\mathbb F_q\)-vector-space direct
factor of \(R/qR\). Hence
\[
\dim_{\mathbb F_q} B\le 3.
\]
On the other hand, \(B\) is an Artinian local ring with residue field \(k\),
so every composition factor of the \(A\)-module \(B\) is isomorphic to \(k\).
Therefore
\[
\dim_{\mathbb F_q} B=[k:\mathbb F_q]\cdot \ell_A(B),
\]
and thus in particular
\[
\ell_A(B)\le 3.
\]

Let \(\mathfrak n:=\mathfrak mA/qA\) be the maximal ideal of \(B\). If \(B=k\),
then
\[
\dim_k\operatorname{Soc}(B)=1.
\]
Otherwise \(\mathfrak n\neq 0\), and then
\[
\operatorname{Soc}(B)=0:_B\mathfrak n\subseteq \mathfrak n.
\]
Indeed, if \(u\in \operatorname{Soc}(B)\setminus \mathfrak n\), then \(u\) is a
unit of \(B\), and \(u\mathfrak n=0\) would force \(\mathfrak n=0\), a
contradiction. Hence \(B/\operatorname{Soc}(B)\neq 0\), so
\[
\ell_A(\operatorname{Soc}(B))\le \ell_A(B)-1\le 2.
\]
Since \(\operatorname{Soc}(B)\) is a \(k\)-vector space, its length equals its
\(k\)-dimension:
\[
\dim_k\operatorname{Soc}(B)=\ell_A(\operatorname{Soc}(B))\le 2.
\]
Therefore
\[
\mathrm{Type}(A)=\dim_k\operatorname{Soc}(B)\le 2.
\]

Since this holds for every maximal ideal \(\mathfrak m\subset R\), every local
ring \(R_{\mathfrak m}\) has Cohen--Macaulay type at most \(2\). Equivalently,
the semilocal ring \(R\) has Cohen--Macaulay type at most \(2\).
\end{proof}

We can expand $h_R$ as a finite Dirichlet Series. We have the following result.

\begin{proposition}
\label{prop: expansion h_R}
Let $R$ be a cubic order in a cubic algebra over $\Q$. For each positive integer $d$ let
\begin{equation*}
    a_R(d) = \begin{cases}
        1, &  $R$\mbox{ is \textbf{not} Gorenstein at every prime dividing $d$ and $d$ is square free,}\\
        0, &  \mbox{ otherwise}.
    \end{cases}
\end{equation*}
Then $a_R(\cdot)$ is a multiplicative arithmetic function and
\begin{equation*}
    h_R(s)=\sum_{d\mid s_{R}}a_R(d)d^{1-s}.
\end{equation*}
\end{proposition}
\begin{proof}
 This follows from the definition.   
\end{proof}

Finally, we are ready to introduce the main object with which we will recover the orbital integrals.

\begin{definition}
\label{def: Order L function }
Let $R_0$ be a $\Z$-order inside a maximal order $\cO$. We define
\begin{equation*}
    L(s, R_0)=\sum_{R_{0}\subset R\subset\mathcal O} h_{R}(s)\, L_R(s)[R: R_{0}]^{1-2s}.
\end{equation*}
and its completion
\begin{equation*}
    \Lambda(s, R_0)=\frac{\Gamma_{R_0,\infty}(s)}{\pi^{-s/2}\Gamma(s/2)}L(s,R_0),
\end{equation*}
where the discriminant of $R_0$ satisfies $D_{R_0}=[\cO: R_0]^2 D_{E/\Q}$ and
\begin{equation*}
    \Gamma_{R_0,\infty}(s)=|D_{R_0}|^{s/2}(\pi^{-s/2}\Gamma(s/2))^{r_1}((2\pi)^{1-s}\Gamma(s))^{r_2}.
\end{equation*}
\end{definition}

\subsubsection{The zeta-function identity and the functional equation}

Let $R$ be an order over $\Z$. Let $R^\vee=\Hom(R, \Z)$ be its dual. Let
\[
J_R(s)=
\sum_{\substack{M\subseteq R^\vee\text{ an }R\text{-submodule}\\
[R^\vee:M]<\infty}}
[R^\vee:M]^{-s}
\]
be Yun's zeta function \cite{ZYun}.

\begin{theorem}\label{teo: identification-to-Yun}
Let $R$ be a Gorenstein cubic order over $\Z$.  Then
we have
\[
J_{R}(s)=L(s, R)\zeta_\Q(s).
\]
\end{theorem}
\begin{proof}
This is \cite[Corollary~9.1]{DengEspinosaYunZeta}.
\end{proof}

\begin{remark}
The identity in Theorem~\ref{teo: identification-to-Yun} appeared as
Conjecture~1 in an earlier version of this paper. Lee \cite[Theorem~1.3]{LeeCubicFE} first proved
unconditionally that, for every Gorenstein order \(R\) in a cubic number
field, the completed overorder zeta function satisfies
\(\Lambda(s,R)=\Lambda(1-s,R)\). Geometrically, the identity
is related to a conjecture of I.~Cherednik~\cite{CH25}.
Arthur gives an explicit algorithmic procedure for constructing a local
orbital \(L\)-function whose value at \(s=1\) is the corresponding orbital
integral \cite{ArthurOrbitalLFunctionsGL3}. However, it is not clear whether
this construction can be adapted to the present application.
\end{remark}

\begin{proposition}
\label{prop: fun eqn}
In the context of the above discussion, \(\Lambda(s,R)\) admits a
meromorphic continuation and satisfies
\begin{equation*}
    \Lambda(s, R)=\Lambda(1-s, R).
\end{equation*}
\end{proposition}
\begin{proof}
By Theorem \ref{teo: identification-to-Yun},
we have \(J_R(s)=L(s,R)\zeta_\Q(s)\). Put
\[
\xi_\Q(s)=\pi^{-s/2}\Gamma(s/2)\zeta_\Q(s).
\]
By the definition of \(\Lambda(s,R)\),
\[
\Gamma_{R,\infty}(s)J_R(s)=\Lambda(s,R)\xi_\Q(s).
\]
Yun's functional equation \cite[Theorem~1.2(2)]{ZYun} and the identity
\(\xi_\Q(s)=\xi_\Q(1-s)\) therefore give
\[
\Lambda(s,R)\xi_\Q(s)
=\Lambda(1-s,R)\xi_\Q(1-s).
\]
Canceling \(\xi_\Q\) as a meromorphic function proves the result.
\end{proof}

\subsubsection{The terms of the regular elliptic part}

The following is the reason we have introduced the above functions. We have

\begin{proposition}
\label{prop: value at 1 finite orbital}
Let $(a, b)\in V$, then
\begin{equation*}
   p^{-k}\cdot L(1, R(a, b)) = \left.\frac{\zeta_E(s)}{\zeta_{\Q}(s)}\right\vert_{s=1} \cdot s^{-1}_{\gamma}\cdot\prod_{q\mid s_{\gamma}} \cO(\gamma, f_q^{p,k}).
\end{equation*}
\end{proposition}
\begin{proof}
The proof of the proposition follows easily from \cite[\S 3.4]{ZYun}. In fact, let 
$R=R(a,b)$ then
we have 
\[
J_R(s)=L(s, R)\zeta_{\Q}(s).
\]
Combining \cite[(3.2) and (3.3)]{ZYun}, we obtain 
\begin{equation*}
p^k\prod_{q\mid s_{\gamma}} \cO(\gamma, f_q^{p,k})=s_{\gamma}\left.\frac{J_R(s)}{\zeta_{E}(s)}\right.\vert_{s=1}, 
\end{equation*}
which implies 
\[
p^{-k}\cdot L(1, R(a, b)) = \left.\frac{\zeta_E(s)}{\zeta_{\Q}(s)}\right\vert_{s=1} \cdot s_{\gamma}^{-1}\cdot\prod_{q\mid s_{\gamma}} \cO(\gamma, f_q^{p,k}).
\]
\end{proof}

\section{The Approximate Functional Equation}

Consider
\[
F(x):=\frac{1}{2K_0(2)}\int_{x}^{\infty} e^{-y-1/y}\frac{dy}{y}
\]
where $K_s(z)$ denotes the $s$-the modified Bessel function of the second kind of order $s$.

Note that we have 
\begin{lemma}
\label{lem: growth of F}
For every $x>0$, we have 
\[
0<F(x)<\frac{e^{-x}}{2K_0(2)}
\]
and
\[
0<1-F(x)<\frac{e^{-1/x}}{2K_0(2)}.
\]
\end{lemma}
\begin{proof}
   We refer the reader to \cite[p.~257]{IK04}.
\end{proof}

Its Mellin transform is
\[
\widetilde F(z)=\int_{0}^{\infty}F(u)u^z\frac{du}{u}=\frac{K_z(2)}{zK_0(2)}.
\]

\begin{lemma}
\label{lem: uniform-bound-tF}
The function $\widetilde F(z)$ is holomorphic except for a simple pole at $z=0$ with residue $1$. Furthermore, 
$\widetilde F(-z)=-\widetilde F(z)$ and for $z=\sigma+it\in \C$, we have uniform bound
\[
\widetilde F(z)\ll |z|^{|\sigma|-1}e^{-(\pi/2)|t|}.
\]
\end{lemma}
\begin{proof}
This follows from the properties of $F_1$ in \cite[pp.~257--258]{IK04}.
\end{proof}

In order to take advantage of Proposition \ref{prop: value at 1 finite orbital} we must invoke the approximate functional equation. This will allow us to express this value as a sum of two \textit{weighted} Dirichlet series. In this way we shall be able to make explicit the dependence of all our values on the parameters $(a, b)$. Furthermore, this process will also introduce rapidly decreasing functions that will smooth out the singularities of the archimedean orbital integrals.

\begin{definition}
\label{def: Dirichlet Series Expansion}
Let $R$ be a cubic order. For a positive integer $n$, we define the coefficients $W_R(n)$ via the Dirichlet series expansion
\[
L_R(s)=\sum_{n=1}^{\infty} \frac{W_R(n)}{n^s},
\]
which holds for $\Re(s)>1$.
\end{definition}
\begin{remark}
    Note that for $E=R\otimes_{\Z} \Q$, the quotient 
    $\frac{\zeta_{E}(s)}{\zeta_{\Q}(s)}$ differs
    from $L_R(s)$ by a finite factor at those $p$
    where $R$ is not maximal, which implies that 
    $L_R(s)$ is convergent for $\Re(s)>1$.
    
\end{remark}

\begin{proposition}
\label{prop: bound-Wn}
For every \(n\ge 1\),
\[
|W_R(n)| \le d(n) \le n.
\]
\end{proposition}

\begin{proof}
Fix a prime \(q\), and write
\[
f_i := [k_i:\mathbb F_q],
\qquad\text{so that}\qquad
|k_i| = q^{f_i}.
\]
Since \(R_q\) is a free \(\mathbb Z_q\)-module of rank \(3\), the quotient
\[
R_q/qR_q
\]
has \(\mathbb F_q\)-dimension \(3\). Every maximal ideal of \(R_q\) contains \(q\), so
\[
qR_q \subseteq J(R_q),
\]
where \(J(R_q)\) is the Jacobson radical. Hence
\[
R_q/J(R_q)\cong \prod_{i=1}^r k_i
\]
is a quotient of \(R_q/qR_q\). Therefore
\[
\sum_{i=1}^r f_i
=
\dim_{\mathbf F_q}\!\left(\prod_{i=1}^r k_i\right)
\le 3.
\]
Thus the multiset \(\{f_i\}\) can only be one of
\[
(1),\ (2),\ (3),\ (1,1),\ (1,2),\ (1,1,1).
\]

Now set \(x=q^{-s}\). Since
\[
\zeta_{\mathbb Q_q}(s)=(1-q^{-s})^{-1}=(1-x)^{-1},
\]
we have
\[
L_{R_q}(s)
=
\frac{\zeta_{R_q}(s)}{\zeta_{\mathbf Q_q}(s)}
=
\frac{1-x}{\prod_i (1-x^{f_i})}.
\]
For the six possible multisets above, this gives
\[
\begin{aligned}
(1)&:\quad L_{R_q}(s)=1,\\
(2)&:\quad L_{R_q}(s)=\frac{1-x}{1-x^2}=\frac{1}{1+x}
      =\sum_{\nu\ge0}(-1)^\nu x^\nu,\\
(3)&:\quad L_{R_q}(s)=\frac{1-x}{1-x^3}
      =\sum_{m\ge0}\bigl(x^{3m}-x^{3m+1}\bigr),\\
(1,1)&:\quad L_{R_q}(s)=\frac{1}{1-x}
      =\sum_{\nu\ge0}x^\nu,\\
(1,2)&:\quad L_{R_q}(s)=\frac{1}{1-x^2}
      =\sum_{m\ge0}x^{2m},\\
(1,1,1)&:\quad L_{R_q}(s)=\frac{1}{(1-x)^2}
      =\sum_{\nu\ge0}(\nu+1)x^\nu.
\end{aligned}
\]
Therefore, if we write
\[
L_{R_q}(s)=\sum_{\nu\ge0} W_R(q^\nu)\,q^{-\nu s},
\]
then in every case
\[
|W_R(q^\nu)|\le \nu+1.
\]

Since
\[
L_R(s)=\prod_q L_{R_q}(s),
\]
the coefficients \(W_R(n)\) are multiplicative, so for
\[
n=\prod_q q^{\nu_q(n)}
\]
we have
\[
W_R(n)=\prod_q W_R\bigl(q^{\nu_q(n)}\bigr).
\]
Hence
\[
|W_R(n)|
\le
\prod_q \bigl(\nu_q(n)+1\bigr)
=
d(n),
\]
where \(d(n)\) is the divisor function. Finally,
\[
d(n)\le n
\qquad (n\ge1),
\]
so
\[
|W_R(n)|\le n.
\]
\end{proof}
\begin{remark}
    In the case of $\GL(2)$ the role of the above Dirichlet Series is played by that of non primitive Hecke Character. In particular, the coefficients are simply $1, -1, 0$, by definition and to bound them is straightforward. 
\end{remark}

The approximate functional equation is the following fact. We will write it for the order $R_{\gamma} (=R(a, b))$ but it holds for cubic orders in general.

\begin{theorem}
\label{thm: approximate functional equation}
The following equation holds:
    \begin{equation*}
        L(1, R_{\gamma}) = \sum_{d|s_{\gamma}}\sum_{R_{\gamma}\subset R\subset\mathcal O_E}\dfrac{1}{[R:R_{\gamma}]}\displaystyle\sum_{n = 1}^{\infty}\dfrac{a_R(d)W_R(n)}{n}\left(F\left(\dfrac{nd[R:R_{\gamma}]^2}{A}\right) +\dfrac{nd[R:R_{\gamma}]^2}{|D_{R_{\gamma}}|^{1/2}}H_{R_{\gamma}}\left(\dfrac{nd[R:R_{\gamma}]^{2}A}{|D_{R_{\gamma}}|}\right)\right)
    \end{equation*}
\end{theorem}
\begin{proof}
    We are going to perform the approximate equation for $L(z, R_{\gamma})$ by using the above functional equation. Fix $z\in\C$ and a positive real number $\sigma$ such that $\sigma + \Re(z) > 1$ and $\Re(z) - \sigma < 0$. For this choice we have the following expansion

\begin{align*}
I &= \frac{1}{2\pi i} \int_{\Re(u) = \sigma}{L}(z + u, R_{\gamma})\widetilde{F}(u)A^udu\\
&= \frac{1}{2\pi i} \int_{\Re(u) = \sigma}\left(\sum_{R_{\gamma}\subset R\subset\mathcal O_E} h_{R}(z+u)\, L_R(z+u)[R: R_{\gamma}]^{1-2z -2u}\right)\widetilde{F}(u)A^udu\\
&= \frac{1}{2\pi i} \int_{\Re(u) = \sigma}\left(\sum_{R_{\gamma}\subset R\subset\mathcal O_E} h_{R}(z+u)\, \sum_{n=1}^{\infty} \frac{W_R(n)}{n^{z+u}}[R: R_{\gamma}]^{1-2z -2u}\right)\widetilde{F}(u)A^udu\\
&= \frac{1}{2\pi i} \int_{\Re(u) = \sigma}\sum_{R_{\gamma}\subset R\subset\mathcal O_E}\sum_{n=1}^{\infty} h_{R}(z+u)  \frac{W_R(n)}{n^{z+u}}[R: R_{\gamma}]^{1-2z -2u}\widetilde{F}(u)A^udu
\end{align*}
We now exchange the sums with the integrals because we are in the region of absolute convergence. We will also take out the factors that depend only on $z$ since they are independent of the variable of integration. We have:
\begin{align*}
I &= \frac{1}{2\pi i} \sum_{R_{\gamma}\subset R\subset\mathcal O_E}\sum_{n=1}^{\infty}\int_{\Re(u) = \sigma} h_{R}(z+u)  \frac{W_R(n)}{n^{z + u}}[R: R_{\gamma}]^{1-2z -2u}\widetilde{F}(u)A^udu\\
&=\frac{1}{2\pi i} \sum_{R_{\gamma}\subset R\subset\mathcal O_E}\sum_{n=1}^{\infty}\frac{W_R(n)}{n^{z}}[R: R_{\gamma}]^{1-2z}
\int_{\Re(u) = \sigma} h_{R}(z+u)  \frac{[R: R_{\gamma}]^{-2u}}{n^u}\widetilde{F}(u)A^udu
\end{align*}
We recall we have also introduced an expansion for $h_R(z + u)$ which we substitute now.
\begin{align*}
I &= \frac{1}{2\pi i} \sum_{R_{\gamma}\subset R\subset\mathcal O_E}\sum_{n=1}^{\infty}\frac{W_R(n)}{n^{z}}[R: R_{\gamma}]^{1-2z}
\int_{\Re(u) = \sigma} \sum_{d|s_{\gamma}} a_R(d)d^{1-z-u} \frac{[R: R_{\gamma}]^{-2u}}{n^u}\widetilde{F}(u)A^udu\\
&= \frac{1}{2\pi i} \sum_{R_{\gamma}\subset R\subset\mathcal O_E}\sum_{n=1}^{\infty}\sum_{d|s_{\gamma}}\frac{W_R(n)}{n^{z}}[R: R_{\gamma}]^{1-2z}a_R(d)d^{1-z}
\int_{\Re(u) = \sigma}  d^{-u} \frac{[R: R_{\gamma}]^{-2u}}{n^u}\widetilde{F}(u)A^udu
\end{align*}
We now invoke the Mellin transform in the integrand. We get
\begin{equation*}
    \frac{1}{2\pi i}\int_{\Re(u) = \sigma}  d^{-u} \frac{[R: R_{\gamma}]^{-2u}}{n^u}\widetilde{F}(u)A^udu = \frac{1}{2\pi i}\int_{\Re(u) = \sigma}  \widetilde{F}(u)\left(Ad^{-1}[R:R_{\gamma}]^{-2}n^{-1}\right)^udu = F\left(\dfrac{nd[R:R_{\gamma}]^2}{A}\right).
\end{equation*}
We conclude that
\begin{equation*}
    I = \sum_{R_{\gamma}\subset R\subset\mathcal O_E}\sum_{n=1}^{\infty}\sum_{d|s_{\gamma}}\frac{W_R(n)}{n^{z}}[R: R_{\gamma}]^{1-2z}a_R(d)d^{1-z}
F\left(\dfrac{nd[R:R_{\gamma}]^2}{A}\right)
\end{equation*}
which upon rearrangement gives
\begin{equation*}
    I = \sum_{d|s_{\gamma}}\dfrac{1}{d^{z-1}}\sum_{R_{\gamma}\subset R\subset\mathcal O_E}a_R(d)[R: R_{\gamma}]^{1-2z}\sum_{n=1}^{\infty}\frac{W_R(n)}{n^{z}}
F\left(\dfrac{nd[R:R_{\gamma}]^2}{A}\right)
\end{equation*}
Shifting the contour from $\Re(u) = \sigma$ to $\Re(u) = \sigma'$ for any $\sigma' < 0$, we get a residue from the pole of $\widetilde{F}(u)$ at $u = 0$.
This pole is simple, so the residue is
$ \lim_{u \rightarrow 0} u \widetilde{F}(u)$.
We deduce the equality
\[
 I = L(z, R_{\gamma}) + \frac{1}{2\pi i} \int_{\Re(u) = \sigma'}{L}(z + u, \gamma)\widetilde{F}(u)A^udu,
\]
We must now evaluate the second integral. We will do this by using the functional equation for the completed function and then repeating an analysis similar to the previous one. Because our order is fixed, we will simply denote the gamma factor by $\Gamma_{\infty}$. Then the functional equation is
\begin{equation*}
    \frac{\Gamma_{\infty}(s)}{\pi^{-s/2}\Gamma(s/2)}L(s,R_{\gamma}) = \frac{\Gamma_{\infty}(1-s)}{\pi^{(s-1)/2}\Gamma((1-s)/2)}L(1 - s,R_{\gamma}).
\end{equation*}
Which in turn implies
\begin{equation*}
    L(1 - s, R_{\gamma}) = \dfrac{1}{\pi^{1/2-s}}\cdot \frac{\Gamma_{\infty}(s)}{\Gamma_{\infty}(1-s)}\cdot \frac{\Gamma((1-s)/2)}{\Gamma(s/2)} \cdot L(s,R_{\gamma})
\end{equation*}
We now recall the reader that 
\begin{equation*}
    \Gamma_{R_{\gamma},\infty}(s)=|D_{R_{\gamma}}|^{s/2}(\pi^{-s/2}\Gamma(s/2))^{r_{1, \gamma}}((2\pi)^{1-s}\Gamma(s))^{r_{2, \gamma}}.
\end{equation*}
Substituting int the above equation we obtain the following
\begin{equation*}
    L(1 - s, R_{\gamma}) = \dfrac{1}{\pi^{1/2-s}}
    \cdot \frac{|D_{R_{\gamma}}|^{s/2}(\pi^{-s/2}\Gamma(s/2))^{r_{1, \gamma}}((2\pi)^{1-s}\Gamma(s))^{r_{2, \gamma}}}{|D_{R_{\gamma}}|^{1-s/2}(\pi^{-(1-s)/2}\Gamma((1-s)/2))^{r_{1, \gamma}}((2\pi)^{s}\Gamma(1-s))^{r_{2, \gamma}}}
    \cdot \frac{\Gamma((1-s)/2)}{\Gamma(s/2)} \cdot L(s,R_{\gamma}).
\end{equation*}

We now introduce a definition that will be important for all our later considerations. For a given order $R$, define
\begin{equation*}
    H_R(s) = \frac{(\pi^{-s/2}\Gamma(s/2))^{r_{1, R}}((2\pi)^{1-s}\Gamma(s))^{r_{2, R}}}{(\pi^{-(1-s)/2}\Gamma((1-s)/2))^{r_{1, R}}((2\pi)^{s}\Gamma(1-s))^{r_{2, R}}}
    \cdot \frac{\Gamma((1-s)/2)}{\Gamma(s/2)}.
\end{equation*}
Using this definition, we obtain
\begin{equation*}
    L(1 - s, R_{\gamma}) = \dfrac{1}{\pi^{1/2-s}}
    \cdot |D_{R_{\gamma}}|^{(2s - 1)/2}\cdot H_{R_{\gamma}}(s) \cdot L(s,R_{\gamma}),
\end{equation*}
or equivalently,
\begin{equation*}
    L(s, R_{\gamma}) = \dfrac{1}{\pi^{s-1/2}}
    \cdot |D_{R_{\gamma}}|^{(1-2s)/2}\cdot H_{R_{\gamma}}(1-s) \cdot L(1-s,R_{\gamma}),
\end{equation*}
We now use this equation and a similar analysis to what we did before. Firstly, notice that we can perform a change of variable $u\mapsto -u$ and get
\begin{equation*}
    \frac{1}{2\pi i} \int_{\Re(u) = \sigma'}{L}(z + u, R_{\gamma})\widetilde{F}(u)A^udu = -\frac{1}{2\pi i} \int_{\Re(u) = -\sigma'}{L}(z - u, R_{\gamma})\widetilde{F}(u)A^{-u}du,
\end{equation*}
where we have used the fact that \(\widetilde F\) is odd. Now, applying the functional equation, we get that the right hand side of integral becomes
\begin{equation*}
   -\frac{1}{2\pi i} \int_{\Re(u) = -\sigma'}\pi^{-z + u+1/2}|D_{R_{\gamma}}|^{(1 -2z + 2u)/2}{L}(1- z + u, R_{\gamma})H_{R_{\gamma}}(1 - z + u)\widetilde{F}(u)A^{-u}du
\end{equation*}
We now substitute the Dirichlet Series Expansion, including the one one for $h_R(s)$. We have that the general summand of this expansion, depending on $R$, $d$  and $n$ is
\begin{equation*}
    \dfrac{a_R(d)W_R(n)[R: R_{\gamma}]^{1-2s}d^{1-s}}{n^s},
\end{equation*}
which evaluated at $1 - z + u$ becomes
\begin{equation*}
    \dfrac{a_R(d)W_R(n)[R: R_{\gamma}]^{-1+2z-2u}d^{z-u}}{n^{1-z+u}},
\end{equation*}
Substituting this into the equation we got of the second integral, we get
\begin{align*}
   \;& -\frac{1}{2\pi i} \sum_{R_{\gamma}\subset R\subset\mathcal O_E}\sum_{n=1}^{\infty}\sum_{d|s_{\gamma}}\int_{\Re(u) = -\sigma'}\pi^{-z + u+1/2}|D_{R_{\gamma}}|^{(1 -2z + 2u)/2}\\    &\hspace{6cm}
         \cdot 
     \dfrac{a_R(d)W_R(n)[R: R_{\gamma}]^{-1+2z-2u}d^{z-u}}{n^{1-z+u}}
       \cdot H_{R_{\gamma}}(1 - z + u)\widetilde{F}(u)A^{-u}du\\
     &=-\pi^{-z+1/2}\frac{1}{2\pi i} \sum_{R_{\gamma}\subset R\subset\mathcal O_E}[R: R_{\gamma}]^{-1+2z}|D_{R_{\gamma}}|^{(1 -2z)/2}\sum_{n=1}^{\infty}\dfrac{1}{n^{1-z}}\sum_{d|s_{\gamma}}a_R(d)W_R(n)d^{z}\int_{\Re(u) = -\sigma'}\pi^{ u}|D_{R_{\gamma}}|^{u}\\
    &\hspace{6cm}
     \cdot 
     \dfrac{[R: R_{\gamma}]^{-2u}d^{-u}}{n^{u}}
     \cdot H_{R_{\gamma}}(1 - z + u)\widetilde{F}(u)A^{-u}du\\
     &=-\pi^{-z+1/2}\frac{1}{2\pi i} \sum_{R_{\gamma}\subset R\subset\mathcal O_E}[R: R_{\gamma}]^{-1+2z}|D_{R_{\gamma}}|^{(1 -2z)/2}\\
    &\hspace{4cm}
     \cdot\sum_{n=1}^{\infty}\dfrac{1}{n^{1-z}}\sum_{d|s_{\gamma}}a_R(d)W_R(n)d^{z}
     \int_{\Re(u) = -\sigma'}
     H_{R_{\gamma}}(1 - z + u)\widetilde{F}(u)\left(\dfrac{A[R:R_{\gamma}]^{2}d n}{\pi|D_{R_{\gamma}}|}\right)^{-u}du.\\
\end{align*}
This leads us to define, depending on a given order $R$,  the function of a real parameter $y$ and complex parameter $z$ by
\begin{equation*}
    H_R(y, z) = \dfrac{\pi^{-z+1/2}}{2\pi i} \int_{\Re(u) =-\sigma'}
     H_{R}(1 - z + u)\widetilde{F}(u)\left(\dfrac{y}{\pi}\right)^{-u}du
\end{equation*}
In this notation, we have that the last sum becomes
\begin{equation*}
   - \sum_{R_{\gamma}\subset R\subset\mathcal O_E}[R: R_{\gamma}]^{-1+2z}|D_{R_{\gamma}}|^{(1 -2z)/2}\sum_{n=1}^{\infty}\dfrac{1}{n^{1-z}}\sum_{d|s_{\gamma}}a_R(d)W_R(n)d^{z}
     H_{R_{\gamma}}\left(\dfrac{A[R:R_{\gamma}]^{2}d n}{|D_{R_{\gamma}}|}, z\right)\\
\end{equation*}
We thus conclude that

\begin{align*}
L(z, R_{\gamma}) = &  \sum_{d|s_{\gamma}}\dfrac{1}{d^{z-1}}\sum_{R_{\gamma}\subset R\subset\mathcal O_E}a_R(d)[R: R_{\gamma}]^{1-2z}\sum_{n=1}^{\infty}\frac{W_R(n)}{n^{z}}
F\left(\dfrac{nd[R:R_{\gamma}]^2}{A}\right)\\ 
&+ \sum_{R_{\gamma}\subset R\subset\mathcal O_E}[R: R_{\gamma}]^{-1+2z}|D_{R_{\gamma}}|^{(1 -2z)/2}\sum_{n=1}^{\infty}\dfrac{1}{n^{1-z}}\sum_{d|s_{\gamma}}a_R(d)W_R(n)d^{z}
     H_{R_{\gamma}}\left(\dfrac{A[R:R_{\gamma}]^{2}d n}{|D_{R_{\gamma}}|}, z\right)
\end{align*}
However, this time we can evaluate at $z = 1$ and the above reduces to the following expression.
\begin{align*}
L(1, R_{\gamma}) = &  \sum_{d|s_{\gamma}}\sum_{R_{\gamma}\subset R\subset\mathcal O_E}a_R(d)[R: R_{\gamma}]^{-1}\sum_{n=1}^{\infty}\frac{W_R(n)}{n}
F\left(\dfrac{nd[R:R_{\gamma}]^2}{A}\right)\\ 
&+ \sum_{R_{\gamma}\subset R\subset\mathcal O_E}[R: R_{\gamma}]|D_{R_{\gamma}}|^{-1/2}\sum_{n=1}^{\infty}\sum_{d|s_{\gamma}}a_R(d)W_R(n)d
     H_{R_{\gamma}}\left(\dfrac{A[R:R_{\gamma}]^{2}d n}{|D_{R_{\gamma}}|}, 1\right).
\end{align*}
Upon rearranging the terms, we get the result. 
\end{proof}

\subsection{The analysis of the functions $H_R(r, z)$}
For a given cubic order, we have led by our computations to the following functions
\begin{align*}
    H_R(s) &= \frac{(\pi^{-s/2}\Gamma(s/2))^{r_{1, R}}((2\pi)^{1-s}\Gamma(s))^{r_{2, R}}}{(\pi^{-(1-s)/2}\Gamma((1-s)/2))^{r_{1, R}}((2\pi)^{s}\Gamma(1-s))^{r_{2, R}}}
    \cdot \frac{\Gamma((1-s)/2)}{\Gamma(s/2)},
\end{align*}
and 
\begin{equation*}
     H_R(r, z) = \dfrac{\pi^{-z+1/2}}{2\pi i} \int_{\Re(u) = -\sigma'}
     H_{R}(1 - z + u)\widetilde{F}(u)\left(\dfrac{r}{\pi}\right)^{-u}du.
\end{equation*}
The pair $(r_1, r_2)$ is either $(3, 0)$ or $(1,1)$. We have used this function for the orders $R_{\gamma}$ associated with the element $\gamma(a, b)$ for some $(a, b)\in V(\pm)$. The pair $(r_1, r_2)$ is completely determined by the type of extension in which this order sits, which in turn simply depends on whether the field $E_{\gamma}$ is totally real or has both real and complex embeddings. As we know, this dichotomy corresponds to $\Pol(a, b) > 0$ or $\Pol(a, b) <0$. Motivated by this fact, we give the following

\begin{definition}
\label{def: function H(x, y)}
    Fix a sign $\pm$.  For each $(x, y)\in D^{\pm}$ we define the function of a complex parameter $s$ by
    \begin{equation*}
        H_{(x, y)}(s) = \begin{cases}
             \pi^{3/2(1 - 2s)}\dfrac{\Gamma(s/2)^2}{\Gamma((1-s)/2)^2}, &  \Pol(x, y) > 0,\\
             \pi^{1/2(1 - 2s)}(2\pi)^{1-2s}\dfrac{\Gamma(s)}{\Gamma(1 - s)}, &  \Pol(x, y) < 0\\
        \end{cases}
    \end{equation*} 
 These are precisely the functions corresponding to the pairs $(3, 0)$ or $(1, 1)$, respectively. Consequently, we also define for each $(x, y)\in D^{\pm}$
 \begin{equation*}
     H_{(x, y)}(r, z) = \dfrac{\pi^{-z+1/2}}{2\pi i} \int_{\Re(u) =-\sigma'}
     H_{(x, y)}(1 - z + u)\widetilde{F}(u)\left(\dfrac{r}{\pi}\right)^{-u}du.
\end{equation*}
\end{definition}
We now prove these functions have good growth behavior.

\begin{proposition}
    \label{prop: growth}
 For $r\ge 1$, we have 
\[
    H_{(x,y)}(r, 1) \ll e^{-3r^{1/3}},
\]
where the implied constant is absolute.
\end{proposition}

\begin{proof}
Set
\[
H_{(x,y)}(r,1)
=\frac{\pi^{-1/2}}{2\pi i}\int_{(1)} H_{(x,y)}(u)\widetilde F(u)\Bigl(\frac{r}{\pi}\Bigr)^{-u}\,du.
\]
Since \(H_{(x,y)}(u)\) is holomorphic for \(\Re(u)>0\) and \(\widetilde F(u)\) has only a simple pole at \(u=0\), we may shift the contour to \(\Re(u)=\sigma\) for any \(\sigma\ge 1\):
\[
H_{(x,y)}(r,1)
=\frac{\pi^{-1/2}}{2\pi i}\int_{(\sigma)} H_{(x,y)}(u)\widetilde F(u)\Bigl(\frac{r}{\pi}\Bigr)^{-u}\,du.
\]

Write \(u=\sigma+it\). By the reflection formula,
\[
H_{(x,y)}(u)=
\begin{cases}
2^{2-2u}\pi^{1/2-3u}\Gamma(u)^2\cos^2(\pi u/2), & \Pol(x,y)>0,\\
2^{1-2u}\pi^{1/2-3u}\Gamma(u)^2\sin(\pi u), & \Pol(x,y)<0.
\end{cases}
\]
Using Stirling's formula uniformly on \(\Re(u)=\sigma\ge 1\), we have
\[
|\Gamma(u)|\ll e^{-\sigma}(1+|u|)^{\sigma-1/2}e^{-\pi|t|/2},
\]
and also
\[
|\cos(\pi u/2)|\ll e^{\pi|t|/2},
\qquad
|\sin(\pi u)|\ll e^{\pi|t|}.
\]
Therefore, in either sign case,
\[
|H_{(x,y)}(\sigma+it)|
\ll e^{-2\sigma}\pi^{-3\sigma}(1+|\sigma+it|)^{2\sigma}.
\]
On the other hand, Lemma~\ref{lem: uniform-bound-tF} gives
\[
|\widetilde F(\sigma+it)|
\ll |\sigma+it|^{\sigma-1}e^{-\pi|t|/2}.
\]
Hence
\[
|H_{(x,y)}(r,1)|
\ll r^{-\sigma}\Bigl(\frac{e^{-2}}{\pi^2}\Bigr)^\sigma
\int_{-\infty}^{\infty}(1+\sigma+|t|)^{3\sigma}e^{-\pi|t|/2}\,dt.
\]

For the \(t\)-integral, after the change of variable \(v=1+\sigma+t\),
\[
\int_0^\infty (1+\sigma+t)^{3\sigma}e^{-\pi t/2}\,dt
= e^{\pi(1+\sigma)/2}\int_{1+\sigma}^\infty v^{3\sigma}e^{-\pi v/2}\,dv
\le e^{\pi(1+\sigma)/2}\Bigl(\frac{2}{\pi}\Bigr)^{3\sigma+1}\Gamma(3\sigma+1).
\]
Thus
\[
|H_{(x,y)}(r,1)|
\ll r^{-\sigma}
\Bigl(\frac{8e^{-2+\pi/2}}{\pi^5}\Bigr)^\sigma
\Gamma(3\sigma+1).
\]
Applying Stirling once more,
\[
\Gamma(3\sigma+1)\ll e^{-3\sigma}(3\sigma)^{3\sigma},
\]
so
\[
|H_{(x,y)}(r,1)|
\ll
\left(
\frac{27\cdot 8\,e^{-5+\pi/2}}{\pi^5}\cdot \frac{\sigma^3}{r}
\right)^\sigma.
\]
Now
\[
\frac{27\cdot 8\,e^{-5+\pi/2}}{\pi^5}<e^{-3},
\]
so if we choose \(\sigma=r^{1/3}\) (which is admissible for \(r\ge 1\)), then
\[
|H_{(x,y)}(r,1)|\ll e^{-3r^{1/3}}.
\]
\end{proof}

\section{The manipulation of the regular elliptic part: Archimedean Orbital Integrals}

The purpose of this section is to see how the approximate functional equation transforms our archimedean orbital integrals into smooth functions.

\subsection{Substitution and manipulation}

Up to this point the regular elliptic part of the trace formula is 
\begin{equation*}
    \displaystyle\sum_{\pm}\displaystyle\sum_{(a, b)\in V(\pm)} \vol(\gamma(a, b))\cO(\gamma(a, b), f^{p,k}),
\end{equation*}
We have
\begin{equation*}
    \vol(\gamma) = |D_{E/\mathbb{Q}}|^{1/2} \cdot\left.\frac{\zeta_E(s)}{\zeta_{\Q}(s)}\right\vert_{s=1} = |\Pol(a, b)|^{1/2} S_{(a, b)}^{-1}\cdot\left.\frac{\zeta_E(s)}{\zeta_{\Q}(s)}\right\vert_{s=1}.
\end{equation*}
Substituting the above, as well as the approximate functional equation, we obtain that the above regular elliptic part becomes

\begin{equation*}
    \displaystyle\sum_{\pm}\displaystyle\sum_{(a, b)\in V(\pm)} L(1, R(a, b))\cdot p^{-k}|\Pol(a, b)|^{1/2} \cdot \cO(\gamma(a, b), f_{\infty}).
\end{equation*}
We want to transform the term 
\begin{equation*}
    |\Pol(a, b)|^{1/2} \cdot \cO(\gamma(a, b), f_{\infty})
\end{equation*}
into the orbital integral in coefficient space we introduced in Definition \ref{def:orbital integral in coefficient space}. However, we must be careful because here we are using the characteristic polynomial
\begin{equation*}
    X^3 - aX^2 + bX \mp p^k,
\end{equation*}
while in coefficient space we imposed the constant term to be $\mp 1$. 

Notice that
\begin{equation*}
    \cO(z\gamma(a, b), f_{\infty}) =  \int_{G_{\gamma}\backslash \GL(3, \R)} f_{\infty}(x^{-1}z\gamma x)dx =  \int_{G_{\gamma}\backslash \GL(3, \R)} f_{\infty}(x^{-1}\gamma x)dx = \cO(\gamma(a, b), f_{\infty}).
\end{equation*}
That is, the orbital integral is also invariant under $Z_+$ and thus we can take instead the representative whose determinant is $\mp 1$. For this representative, the characteristic polynomial is 
\begin{equation*}
    X^3 - \dfrac{a}{p^{k/3}}X^2 + \dfrac{b}{p^{2k/3}}X \mp 1.
\end{equation*}
Recall that
\begin{equation*}
    \Pol(za, z^2b, z^3c) = z^6\Pol(a, b, c).
\end{equation*}
In particular,
\begin{equation*}
    \Pol(a, b) = \Pol(a, b, \mp p^k) = p^{2k}\Pol(ap^{-k/3}, bp^{-2k/3}, \mp 1).
\end{equation*}
Putting all together, and recalling the function $\theta^{\pm}$ from Definition \ref{def:orbital integral in coefficient space}, we get
\begin{equation*}
    |\Pol(a, b)|^{1/2}\cO(\gamma(a, b), f_{\infty}) = p^{k}|\Pol(ap^{-k/3}, bp^{-2k/3}, \mp 1)|^{1/2}\cO(p^{-k/3}\gamma(a, b), f_{\infty}) = p^k\theta^{\pm}(ap^{-k/3}, bp^{-2k/3}).
\end{equation*}
As a consequence, we have that the regular elliptic part is the following sum.
\begin{equation}
\label{eqn: reg ell part with theta}
    \displaystyle\sum_{\pm}\displaystyle\sum_{(a, b)\in V(\pm)} L(1, R(a, b))\,\theta^{\pm}(ap^{-k/3}, bp^{-2k/3}).
\end{equation}
Notice that the factor $p^k$ cancelled.
\subsection{Smoothing of singularities}

We now substitute in equation \eqref{eqn: reg ell part with theta} the approximate functional equation. We get that the regular elliptic part is the sum of
 \begin{equation*}
    \displaystyle\sum_{\pm}\displaystyle\sum_{(a, b)\in V(\pm)} \sum_{d|s_{\gamma}}\sum_{R_{\gamma}\subset R\subset\mathcal O_E}a_R(d)[R: R_{\gamma}]^{-1}\sum_{n=1}^{\infty}\frac{W_R(n)}{n}
F\left(\dfrac{nd[R:R_{\gamma}]^2}{A}\right)\theta^{\pm}(ap^{-k/3}, bp^{-2k/3}).
\end{equation*}   
and
 \begin{equation*}
    \displaystyle\sum_{\pm}\displaystyle\sum_{(a, b)\in V(\pm)} \sum_{R_{\gamma}\subset R\subset\mathcal O_E}[R: R_{\gamma}]|D_{R_{\gamma}}|^{-1/2}\sum_{n=1}^{\infty}\sum_{d|s_{\gamma}}a_R(d)W_R(n)d\,
     H_{R_{\gamma}}\left(\dfrac{A[R:R_{\gamma}]^{2}d n}{|D_{R_{\gamma}}|}, 1\right)\theta^{\pm}(ap^{-k/3}, bp^{-2k/3}).
\end{equation*}   

Motivated by these expressions we give the following definitions in $D^{\pm}$. 

\begin{definition}
\label{def: Schwartz function}
Let $\pm$ be a sign, $f$ and $N$ be positive integers and fix a real number $0 < \alpha < 1$. Define the functions $\Phi^{\pm}_{N, f}:D^{\pm} \longrightarrow \C$ and $\Psi^{\pm}_{N, f}:D^{\pm} \longrightarrow \C$ by
\begin{align*}
    \Phi^{\pm}_{N, f}(x, y) &= \theta^{\pm}(xp^{-k/3}, yp^{-2k/3})F\left(\dfrac{Nf^2}{|\Pol(x, y)|^{\alpha}}\right),\\
    \Psi^{\pm}_{N, f}(x, y) &= \theta^{\pm}(xp^{-k/3}, yp^{-2k/3})|\Pol(x, y)|^{-1/2}H_{(x, y)}\left(\dfrac{Nf^{2}}{|\Pol(x, y)|^{1 - \alpha}}, 1\right),
\end{align*}
where, for the corresponding choice of sign, $\Pol(x, y) = \Pol(x, y, \pm p^k)$.
\end{definition}
\begin{remark}
These functions depend implicitly on the parameter $\alpha$. However, since $\alpha$ is fixed for us through out the paper we omit the dependence. 
\end{remark}

In order to see how the above simplifies our manipulations we must fix a value for $A$ in the approximate functional equation. The result is as follow.
\begin{proposition}
The regular elliptic part of the trace formula is
\begin{equation*}
     \displaystyle\sum_{\pm}\displaystyle\sum_{(a, b)\in V(\pm)} \sum_{d|s_{\gamma}}\sum_{R_{\gamma}\subset R\subset\mathcal O_E}a_R(d)[R: R_{\gamma}]^{-1}\sum_{n=1}^{\infty}\frac{W_R(n)}{n}\left(\Phi^{\pm}_{nd, [R:R_{\gamma}]}(a, b) + nd[R:R_{\gamma}]^2\Psi^{\pm}_{nd, [R:R_{\gamma}]}(a, b)\right).
\end{equation*}
\end{proposition}
\begin{proof}
    This follows by picking in Theorem \ref{thm: approximate functional equation} the value $A = |\Pol(a, b)|^{\alpha}$ and noticing that $D_{R_{\gamma}} = \Pol(a, b)$ because the discriminant of the order generated by $\gamma$ is precisely the ideal generated by the discriminant of its characteristic polynomial.
\end{proof}

The following result settles one of the hypothesis for performing Poisson Summation: the involved function must be smooth and compactly supported

\begin{theorem}
\label{thm: Schwartz Function}
Let $\pm$ be a sign, $f$ and $N$ be positive integers and fix a real number $0 < \alpha < 1$. The extension by zero to $\mathbb{R}^2$ of the functions $\Phi^{\pm}_{N, f}$ and $\Psi^{\pm}_{N, f}$ are smooth and compactly supported.
\end{theorem}
\begin{proof}
  Luckily for us this has already been settled in \cite{GKMPW}. However, we wish to emphasize the fact that the extension to the singular set must be by zero.  We sketch the argument for $\Phi^{\pm}_{N, f}$. 
  
  Fix $M\ge 0$ and let $z = (x_0, y_0)\in\R^2$ be a singular element. This means $\Pol(x_0, y_0, \pm p^k) = 0$. The continuity of $\Pol$ guarantees there exists a neighborhood around $z$ such that
  \begin{equation*}
      \dfrac{Nf^2}{|\Pol(x, y)|^{\alpha}} > M.
  \end{equation*}
  Lemma \ref{lem: growth of F} implies
  \begin{equation*}
      F\left(\dfrac{Nf^2}{|\Pol(x, y)|^{\alpha}}\right) << e^{-M}.
  \end{equation*}
  Furthermore, $\theta^{\pm}$ is the \textit{normalized} orbital integral. Thus, around a central point it is bounded (the bound of course might depend on $z$). We conclude that around $z$,
  \begin{equation*}
      \Phi^{\pm}_{N, f}(x, y) << e^{-M}.
  \end{equation*}
  Letting $M\longrightarrow \infty$ we conclude the limit at $z$ exists and its zero. This proves that the extension by zero is indeed continuous. Smoothness follows a more convoluted argument keeping track of partial derivatives of increasing orders.
\end{proof}
\begin{remark}
    A reader familiar with \cite{AliI} or \cite{BEgenfields} will realize that in those papers the corresponding functions $\theta^{\pm}$ are defined by invoking the germ expansions around central elements. This is an innocent abuse of notation as long as one realizes why. 
    Germ expansions exist around any central element but are not uniform. In particular, one cannot really define via the germ expansions the normalized orbital integral beyond their domain of definition. That is, very deep in the regular set a given germ expansion might not hold. However, in the regular set one does not need the germ expansion for anything because there the normalized orbital integral is smooth.
    The germ expansion are only really required to prove the smoothness of the corresponding functions $\Phi^{\pm}_{N, f}$  and $\Psi^{\pm}_{N, f}$ at the central elements. This explains why we have not introduced germ expansions at all. Their use is concentrated in the proof of the above theorem which we refer the reader to \cite{GKMPW} to see.
\end{remark}

\section{Poisson Summation}

Our goal in this section is to complete the coefficient lattice, apply
Poisson summation, and isolate the resulting zero-frequency contribution.
To this end, we first identify the complete lattice on which Poisson
summation may be applied.  To preserve the flow of the argument, we state
the necessary results and describe their consequences before giving their
proofs.

\subsection{The problem of completion}

The regular elliptic expression obtained above is
\[
\sum_{\epsilon=\pm1}\sum_{(a,b)\in V(\epsilon)}
\sum_{d\mid s_{(a,b)}}\sum_{R(a,b)\subseteq R\subseteq\mathcal O_E}
\frac{a_R(d)}{[R:R(a,b)]}
\sum_{m\ge1}\frac{W_R(m)}m
\left(
\Phi^\epsilon_{md,[R:R(a,b)]}(a,b)
+md[R:R(a,b)]^2\Psi^\epsilon_{md,[R:R(a,b)]}(a,b)
\right).
\]
This series is absolutely convergent. Indeed, the support of
\(\theta^\epsilon\) contains only finitely many lattice points
\((a,b)\), and each corresponding order has only finitely many
overorders. The bounds \(|a_R(d)|\le1\),
\(|W_R(m)|\le m\), Lemma~\ref{lem: growth of F}, and
Proposition~\ref{prop: growth} then give rapid convergence in \(m\).
We may therefore reorganize the series freely.

\begin{definition}
\label{def: extension of orders}
Fix \(\epsilon\in\{\pm1\}\).  For every \((a,b)\in\Z^2\), define
\[
P^\epsilon_{a,b}(X):=X^3-aX^2+bX-\epsilon p^k,\qquad
\Delta^\epsilon(a,b):=\operatorname{disc}(P^\epsilon_{a,b}),\qquad
R^\epsilon(a,b):=\Z[X]/(P^\epsilon_{a,b}(X)).
\]
An \emph{over-ring} of \(R^\epsilon(a,b)\) means a unital subring with
the same identity,
\[
R^\epsilon(a,b)\subseteq R\subseteq
R^\epsilon(a,b)\otimes_\Z\Q,
\]
of finite index.  When the sign is fixed, we continue to suppress
\(\epsilon\) and write \(R(a,b)\).
\end{definition}

Because \(P^\epsilon_{a,b}\) is monic of degree \(3\),
\(R^\epsilon(a,b)\) is a finite free cubic ring for every \((a,b)\).
If \(\Delta^\epsilon(a,b)\ne0\), then
\(R^\epsilon(a,b)\otimes_\Z\Q\) is an \'etale cubic algebra. For a
reducible polynomial it is a product of number fields, and
\(R^\epsilon(a,b)\) is an ordinary order in that product.  If
\(\Delta^\epsilon(a,b)=0\), the same quotient still defines a finite free cubic
ring, although \(R^\epsilon(a,b)\otimes_\Z\Q\) is no longer \'etale.
We do not choose a maximal order at such a point.  For a fixed \(f\),
the set of index-\(f\) over-rings is finite:
every such ring lies between \(R(a,b)\) and \(f^{-1}R(a,b)\).
The definitions of \(a_R(d)\) and \(W_R(n)\) use only the local finite
flat cubic rings \(R\otimes\Z_q\), so they apply to these over-rings as
well: \(W_R(q^\nu)\) is obtained from the coefficients of the local
quotient \(\zeta_{R_q}(s)/\zeta_{\Q_q}(s)\), while \(a_R(q)\) records
whether \(R_q\) is non-Gorenstein and \(a_R(q^t)=0\) for \(t\ge2\).
We do not assert a global functional equation or apply the
approximate functional equation to a singular cubic ring. Such argument
has already been carried out on \(V(\epsilon)\).  If
\((a,b)\in V(\epsilon)\), this construction agrees with the original
order in the cubic field, and its finite-index over-rings are the usual
overorders contained in the ring of integers \(\mathcal O_{(a,b)}\) of \(R^\epsilon(a,b)\otimes_\Z\Q\).

\subsection{Periodicity and Poisson Summation}

\begin{proposition}
\label{prop: d not divide s vanishes}
Fix \((a,b)\in V(\epsilon)\) and
\(R(a,b)\subseteq R\subseteq\mathcal O_E\). If
\(d\nmid s_{(a,b)}\), then \(a_R(d)=0\).
\end{proposition}
\begin{proof}
Choose \(q\) with \(v_q(d)>v_q(s_{(a,b)})\). If \(v_q(d)>1\),
then \(d\) is not squarefree and \(a_R(d)=0\) by definition. The
remaining case is \(v_q(d)=1\) and \(v_q(s_{(a,b)})=0\). Then
\(R(a,b)_q\), and hence \(R_q\), is maximal and Gorenstein, so the
definition again gives \(a_R(d)=0\).
\end{proof}

For every \((a,b)\in\Z^2\), define
\[
S^\epsilon(a,b;N,f):=
\sum_{\substack{R\text{ an over-ring of }R^\epsilon(a,b)\\
{}[R:R^\epsilon(a,b)]=f}}
\sum_{d\mid N}d\,a_R(d)W_R(N/d).
\]
On the regular elliptic locus, the preceding proposition followed by
the change of variables \(N=md\) rewrites the original expression as
\[
\Iel(f^{p,k})=
\sum_{\epsilon=\pm1}\sum_{N,f\ge1}\frac1{Nf}
\sum_{(a,b)\in V(\epsilon)}
S^\epsilon(a,b;N,f)\,
\Theta^\epsilon_{N,f}(a,b),
\]
where
\[
\Theta^\epsilon_{N,f}
:=\Phi^\epsilon_{N,f}+Nf^2\Psi^\epsilon_{N,f}.
\]

\begin{theorem}
\label{thm: periodicity}
Fix \(N,f\ge1\). If \((a,b),(a',b')\in\Z^2\) and
\((a,b)\equiv(a',b')\pmod{Nf^2}\), then
\[
S^\epsilon(a,b;N,f)=S^\epsilon(a',b';N,f).
\]
\end{theorem}
The proof is given in Subsection~\ref{subsec: periodicity}.

Thus \(S^\epsilon(a,b;N,f)\) itself is a well-defined periodic
coefficient on the whole lattice; no second notation is required.
Independently, Theorem~\ref{thm: Schwartz Function} extends
\(\Theta^\epsilon_{N,f}\) smoothly by zero across the
discriminant-zero locus.

\begin{definition}
\label{def: completed regular elliptic part}
Following Altu\u{g}'s notation  in
\cite[Theorem~4.2]{AliI}, define the contribution of the added lattice
points by
\[
\Sigma(f^{p,k}):=
\sum_{\epsilon=\pm1}\sum_{N,f\ge1}\frac1{Nf}
\sum_{(a,b)\in\Z^2\setminus V(\epsilon)}
S^\epsilon(a,b;N,f)\Theta^\epsilon_{N,f}(a,b).
\]
Then
\[
\Iel(f^{p,k})+\Sigma(f^{p,k})=
\sum_{\epsilon=\pm1}\sum_{N,f\ge1}\frac1{Nf}
\sum_{(a,b)\in\Z^2}
S^\epsilon(a,b;N,f)
\Theta^\epsilon_{N,f}(a,b).
\]
\end{definition}

Here ``completion'' means only that the summation set has been enlarged
from \(V(\epsilon)\) to \(\Z^2\).  The arithmetic factor at each added point is computed
from the actual cubic ring in Definition~\ref{def: extension of orders}.
If \(\Delta^\epsilon(a,b)=0\), this ring and its over-ring sum are still defined,
but \(\Theta^\epsilon_{N,f}(a,b)=0\) by the continuous extension in
Theorem~\ref{thm: Schwartz Function}; hence the corresponding full
summand is zero.
The support of \(\theta^\epsilon\) contains only finitely many lattice
points.  The discriminant-zero points contribute zero.  At every
remaining point the generic algebra is \'etale, and all over-rings lie
in its integral closure. Hence there are only finitely many over-rings
in total.  Absolute convergence now follows from
\(|a_R(d)|\le1\), \(|W_R(m)|\le m\), the rapid decay in
Lemma~\ref{lem: growth of F}, and Proposition~\ref{prop: growth}.

Put \(M:=Nf^2\).  For fixed
\(\epsilon\in\{\pm1\}\) and \(N,f\geq1\), regard
\(S^\epsilon(\,\cdot\,,\,\cdot\,;N,f)\) as a function on
\((\Z/M\Z)^2\).  Decomposing \(\Z^2\) into cosets of
\(M\Z^2\), we obtain
\begin{align}
\label{eq:lattice-coset-decomposition}
\Iel(f^{p,k})+\Sigma(f^{p,k})
&=
\sum_{\epsilon=\pm1}\sum_{N,f\geq1}\frac1{Nf}
\sum_{(a_0,b_0)\in(\Z/M\Z)^2}
S^\epsilon(a_0,b_0;N,f) \notag\\
&\qquad\qquad{}\times
\sum_{\substack{(a,b)\in\Z^2\\
(a,b)\equiv(a_0,b_0)\!\!\!\pmod M}}
\Theta^\epsilon_{N,f}(a,b).
\end{align}
The interchange of the sums is justified by the absolute convergence
established above.

Let \(e(t):=e^{2\pi i t}\), and normalize the Euclidean Fourier
transform by
\[
\widehat{\Theta}^\epsilon_{N,f}(u,v)
:=
\int_{\R^2}\Theta^\epsilon_{N,f}(x,y)e(-xu-yv)\,dx\,dy.
\]
By Theorem~\ref{thm: Schwartz Function},
\(\Theta^\epsilon_{N,f}\), extended by zero, belongs to
\(C_c^\infty(\R^2)\) and is therefore a Schwartz function.  Poisson
summation on the coset \((a_0,b_0)+M\Z^2\) gives
\begin{align}
\label{eq:poisson-on-coset}
&\sum_{\substack{(a,b)\in\Z^2\\
(a,b)\equiv(a_0,b_0)\!\!\!\pmod M}}
\Theta^\epsilon_{N,f}(a,b) \notag\\
&\qquad=
\frac1{M^2}\sum_{(\xi,\eta)\in\Z^2}
e\left(\frac{a_0\xi+b_0\eta}{M}\right)
\widehat{\Theta}^\epsilon_{N,f}
\left(\frac{\xi}{M},\frac{\eta}{M}\right).
\end{align}

Define the finite Fourier transform of the arithmetic coefficient by
\[
\mathcal K^\epsilon_{N,f}(\xi,\eta)
:=
\sum_{(a_0,b_0)\in(\Z/M\Z)^2}
S^\epsilon(a_0,b_0;N,f)
e\left(\frac{a_0\xi+b_0\eta}{M}\right).
\]
Substituting \eqref{eq:poisson-on-coset} into
\eqref{eq:lattice-coset-decomposition} and using
\(M^2=N^2f^4\), we find
\begin{equation}
\label{eq:poisson-expanded-sum}
\Iel(f^{p,k})+\Sigma(f^{p,k})
=
\sum_{\epsilon=\pm1}\sum_{N,f\geq1}
\frac1{N^3f^5}
\sum_{(\xi,\eta)\in\Z^2}
\mathcal K^\epsilon_{N,f}(\xi,\eta)
\widehat{\Theta}^\epsilon_{N,f}
\left(\frac{\xi}{Nf^2},\frac{\eta}{Nf^2}\right).
\end{equation}

At the zero dual frequency,
\[
K^\epsilon(N,f)
:=\mathcal K^\epsilon_{N,f}(0,0)
=\sum_{(a_0,b_0)\in(\Z/Nf^2\Z)^2}
S^\epsilon(a_0,b_0;N,f).
\]
Consequently, the total zero-frequency contribution is
\begin{equation}
\label{eqn: main part}
\sum_{\epsilon=\pm1}\sum_{N,f\geq1}
\frac{K^\epsilon(N,f)}{N^3f^5}
\int_{\R^2}\Theta^\epsilon_{N,f}(x,y)\,dx\,dy.
\end{equation}
The contribution of the nonzero dual frequencies is
\begin{align}
\label{eq:nonzero-frequency-contribution}
\Sigma((\xi,\eta)\ne(0,0))
:={}&
\sum_{\epsilon=\pm1}\sum_{N,f\geq1}\frac1{N^3f^5}
\sum_{\substack{(\xi,\eta)\in\Z^2\\(\xi,\eta)\ne(0,0)}}
\mathcal K^\epsilon_{N,f}(\xi,\eta) \notag\\
&\qquad{}\times
\widehat{\Theta}^\epsilon_{N,f}
\left(\frac{\xi}{Nf^2},\frac{\eta}{Nf^2}\right).
\end{align}
Thus \(\Iel(f^{p,k})+\Sigma(f^{p,k})\) is the sum of
\eqref{eqn: main part} and
\(\Sigma((\xi,\eta)\ne(0,0))\).

\begin{remark}\label{rem: period-GL2}
The same periodicity conclusion applies to the \(\GL(2, \Q)\) calculation of Altu\u{g} \cite{AliI}. Although the Kronecker-symbol formulation there uses the period \(4nf^2\), the congruence condition and the symbol together are already invariant modulo \(nf^2\). Hence the trace lattice may be decomposed into classes modulo \(nf^2\), consistently with the uniform period above.
\end{remark}

\section{The Kloosterman Sums for
\texorpdfstring{\(\GL(3,\mathbb{Q})\)}{GL(3,Q)}}
\label{sec: The Kloosterman Sums}

\subsection{Definition and main properties}

We now analyze the arithmetic sums arising from the zero frequency.

For \(n,f\ge1\) and a sign \(\epsilon\), recall the zero-frequency sum
\begin{equation}\label{eqn: Kloosterman Type Sums}
K^\epsilon(n,f)
:=\sum_{(a,b)\,(\mathrm{mod}\,nf^2)}
S^\epsilon(a,b;n,f).
\end{equation}
Using this definition, the main part becomes
\begin{equation*}
        \displaystyle\sum_{\pm}
        \sum_{n=1}^{\infty}
        \sum_{f=1}^{\infty}\dfrac{K^{\pm}(n, f)}{n^3f^5}
        \displaystyle\int_{\R^2}\Theta_{n, f}^{\pm}(x, y)dxdy.
\end{equation*}
The function $\Theta_{n, f}^{\pm}$ includes in its definition the functions $\Phi^{\pm}_{n, f}$ and $\Psi^{\pm}_{n, f}$. These in turn are defined by certain complex integrations over vertical lines $\Re(z) = \sigma$. Our desire is to bring those complex integrals outside all of the above sums and perform calculus of residues. This requires that we introduce a function of a complex variable whose special values recover the sums over $n$ and $f$ above.

\begin{definition}
\label{def: Dirichlet Series}
    Fix a sign $\pm$. We define the  Dirichlet series
    \begin{equation*}
        \D^{\pm}(z) = \sum_{n=1}^{\infty}
        \sum_{f=1}^{\infty}\dfrac{K^{\pm}(n, f)}{n^{z + 2}f^{2z + 3}}.
    \end{equation*}
\end{definition}
\begin{remark}
    The sign in \(K^\pm(n,f)\) and \(\D^\pm(z)\) determines the
    ring \(R^\pm(a,b)\) through the polynomial
    \(X^3-aX^2+bX\mp p^k\). The sums also depend on the fixed prime
    \(p\), which is suppressed from the notation. Below we prove that
    \[
    K^{-}(n,f)=K^{+}(n,f)
    \]
    which explains the reason we sometimes only write $K(n,f)$.
\end{remark}

The following is the main result of this section and one of the fundamental results of the present paper.

\begin{theorem}
\label{thm: Dirichlet Sums Evaluation}
For a complex parameter $z$ with $\Re(z) > 1$, we have
\begin{equation*}
    \D^{\pm}(z) = \dfrac{\zeta(2z)}{\zeta(z+1)} \cdot \dfrac{\zeta(3z)}{\zeta(2z+1)} \cdot \dfrac{1 - p^{-z(k+1)}}{1 - p^{-z}} \cdot \dfrac{1 - p^{-z(k+2)}}{1 - p^{-2z}}
\end{equation*}
In particular, $\D^{\pm}(z)$ is independent of the sign $\pm$ and admits a meromorphic continuation to the whole complex plane.
\end{theorem}
\begin{proof}
For each prime \(q\), the local sum in
\cite[Section~10]{DengEspinosaYunZeta} evaluates every residue class
at a \(\Z_q\)-representative of nonzero discriminant.  Such a
representative exists in every \(q\)-adic residue class, because the
discriminant is a nonzero polynomial in \(a,b\).  Lemma~\ref{lem:
strong local periodicity} shows that our directly defined
\(S_q(a,b;v,r)\), including at a singular representative, has the same
value.  Hence our local \(K_q(q^v,q^r)\) is exactly the one used there.
Proposition~\ref{prop: factorization-kloosterman} identifies the global
series with the product of these local series, and
\cite[Corollary~10.7]{DengEspinosaYunZeta} gives the stated formula.
\end{proof}

\subsection{Local factorizations}
Fix \(\epsilon\in\{\pm1\}\) and suppress it from
\(R^\epsilon(a,b)\), \(S^\epsilon\), and \(K^\epsilon\) throughout
the remainder of this section.  We will prove that the Kloosterman sums admit an
Euler product decomposition, beginning with their local factors.

\begin{definition}\label{def: local-kloosterman-p}
Let \(q\) be a prime. Define
\begin{equation}\label{eqn: local-kloosterman-p}
K_q(q^v, q^r)=\sum_{(a, b)\mod q^{v+2r}}\sum_{\substack{R\supseteq R(a,b)\\ [R: R(a,b)]=q^r}}\sum_{0\le t\leq v}q^t a_R(q^t)W_R(q^{v-t}).
\end{equation}
Here the inner sum is local: \(R(a,b)\) denotes
\(R(a,b)\otimes\Z_q\), and \(R\) runs over its index-\(q^r\)
over-rings in \(R(a,b)\otimes\Q_q\).
\begin{remark}
    Notice that this definition holds for all primes $q$. If our period had been $4nf^2$ then at $q = 2$ we would have to make a special definition to take into account the $4$. A comparison with \cite{AliI} and \cite{BEgenfields} shows in these papers this distinction was made.
\end{remark}  
\end{definition}

Throughout this section, a residue class \((a,b)\) may be represented
by any integral pair, and \(R(a,b)\) is the cubic ring of
Definition~\ref{def: extension of orders} attached to that pair.  The
coefficient \(S^\pm(a,b;n,f)\), rather than the individual presentation
of the ring, is independent of the representative by
Theorem~\ref{thm: periodicity}.

Before we prove the local factorization we provide some useful notation and results. Based on the definitions of $K(n, r)$ and $K_q(q^v, q^ r)$ we introduce the following sums
\begin{align*}
    \Sigma_R(n) &:=\sum_{d\mid n} a_R(d)\,W_R(n/d)\,d,\\
    \Sigma_{R, q}(q^v) &:=\sum_{t = 0}^v a_R(q^t)\,W_R(q^{v - t})\,q^t,\\
    S(a,b; n, f) &:=\sum_{\substack{R\supseteq R(a,b)\\ [R:R(a,b)]=f}}\sum_{d|n}d a_R(d)W_R(n/d) = \sum_{\substack{R\supseteq R(a,b)\\ [R:R(a,b)]=f}} \Sigma_R(n),\\
    S_q(a,b;v, r)&:=\sum_{\substack{R\supseteq R(a,b)\\ [R:R(a,b)]=q^r}}\sum_{0\le t\le v}q^t a_R(q^t)W_R(q^{v-t}) = \sum_{\substack{R\supseteq R(a,b)\\ [R:R(a,b)]=q^r}}\Sigma_{R, q}(q^v).
\end{align*}
By definition, we know
\begin{align*}
    K^{\pm}(n, f) &= \displaystyle\sum_{(a, b) \mod nf^2} S(a, b; n, f),\\
     K_q(q^v, q^r) &= \displaystyle\sum_{(a, b) \mod q^{v + 2r}} S_q(a, b; v, r)\\
\end{align*}
Our goal is to verify the multiplicativity of the Kloosterman type sums. The strategy is to prove this mutliplicativity in stages: first for $\Sigma$, then $S$ and finally $K$. 

Fix a pair \((a,b)\) and put \(R_0:=R(a,b)\).  Let
\(\mathcal O_f(R_0)\) be the finite set of over-rings
\(R_0\subseteq R\subseteq R_0\otimes\Q\) with \([R:R_0]=f\).

We can prove the multiplicativity of $\Sigma$ in a straightforward way.

\begin{lemma}\label{lemma: multiplicativity of sigma}
  Let $n=\displaystyle\prod_{q}q^{e_q}$, then 
  \begin{equation*}
      \Sigma_R(n) = \displaystyle\prod_q \Sigma_{R_q,q}(q^{e_q}).
  \end{equation*}
\end{lemma}
\begin{proof}
By their local definitions, the coefficients \(a_R(\cdot)\) and
\(W_R(\cdot)\) are multiplicative and factor through the local finite
flat rings \(R_q\).
Therefore, writing $d=\prod_q q^{t_q}$ with $0\le t_q\le e_q$, we obtain the Euler factorization
\begin{equation}\label{eq:Sigma-factor}
\Sigma_R(n)
=
\prod_q \left(\sum_{t=0}^{e_q} a_R(q^t)\,W_R(q^{e_q-t})\,q^{t}\right)
=
\prod_q \left(\sum_{t=0}^{e_q} a_{R_q}(q^t)\,W_{R_q}(q^{e_q-t})\,q^{t}\right)
=
\prod_q \Sigma_{R, q}(q^{e_q}).
\end{equation}
\end{proof}

The multiplicativity of $S$ relies on that of its indexing set. This is given by the following result.

\begin{lemma}\label{lem:overorders-local-global}
Write \(f=\prod_q q^{r_q}\).
The map
\[
\mathcal O_f(R_0)\longrightarrow \widehat{\prod_q} \mathcal O_{q^{r_q}}\bigl((R_0)_q\bigr),
\qquad
R\longmapsto R_q,\ \ R_q:=R\otimes\Z_q,
\]
is a bijection. On the right, the restricted product is taken with
respect to \((R_0)_q\); thus the chosen over-ring equals \((R_0)_q\)
at all but finitely many primes. Moreover,
\[
[R:R_0]=\prod_q [R_q:(R_0)_q].
\]
\end{lemma}

\begin{proof}
First, if $R\in\mathcal O_f(R_0)$ then $R/R_0$ is a finite abelian group of order $f$, hence
\[
(R/R_0)\otimes\Z_q \cong R_q/(R_0)_q
\]
is a finite $\Z_q$-module of cardinality $q^{r_q}$. This gives $[R_q:(R_0)_q]=q^{r_q}$ and the
product formula for $[R:R_0]$.

Injectivity: if $R,R'\supseteq R_0$ satisfy $R\otimes\Z_q=R'\otimes\Z_q$ for every $q$, then
\[
R=\{x\in R_0\otimes\Q:\ x\in R\otimes\Z_q
\text{ inside }R_0\otimes\Q_q\text{ for every }q\},
\]
and the analogous identity holds for \(R'\).  Equality of all local
lattices therefore gives \(R=R'\).

Surjectivity: let $(S_q)_q$ be a tuple with $S_q\supseteq (R_0)_q$ an over-ring of index $q^{r_q}$
for each $q$, and $S_q=(R_0)_q$ for all but finitely many $q$. Define
\[
R:=\{x\in R_0\otimes\Q:\ x\in S_q\ \text{inside }(R_0\otimes\Q_q)\ \text{for every }q\}.
\]
Then $R$ is a ring containing $R_0$, and by construction $R\otimes\Z_q=S_q$ for every $q$.
Furthermore, $R/R_0$ is finite and its $q$-primary part is $S_q/(R_0)_q$, hence $[R:R_0]=\prod_q q^{r_q}=f$.
Thus $R\in\mathcal O_f(R_0)$ and maps to $(S_q)_q$.
\end{proof}

As a consequence of the above result we get
\begin{lemma}\label{lem: multiplicativity of S}
Let $n=\displaystyle\prod_{q}q^{e_q}, f=\displaystyle\prod_{q}q^{r_q}$, then we have 
\begin{equation*}
    S(a, b; n, f) = \displaystyle\prod_q S_q(a, b; e_q, r_q).
\end{equation*}
\end{lemma}
\begin{proof}
    By definition,
\[
S(a, b;n, f)
=
 \sum_{\substack{R\supseteq R(a,b)\\ [R:R(a,b)]=f}}
\Sigma_R(n).
\]
Now apply Lemma~\ref{lem:overorders-local-global} to identify the
index-\(f\) over-rings \(R\) with tuples of local over-rings
\((R_q)_q\) of indices \(q^{r_q}\). Thus
\begin{align*}
S(a,b;n,f)
&=
\sum_{(R_q)_q}
\prod_q\Sigma_{R_q,q}(q^{e_q})\\
&=
\prod_q
\sum_{R_q\in\mathcal O_{q^{r_q}}((R_0)_q)}
\Sigma_{R_q,q}(q^{e_q})
=\prod_q S_q(a,b;e_q,r_q).
\end{align*}
Only finitely many factors differ from \(1\), so the product and the
interchange of the finite sums are unambiguous.
\end{proof}

We are ready to verify the local Euler decomposition of the Kloosterman type sums.

\begin{proposition}\label{prop: factorization-kloosterman}
Let $n=\displaystyle\prod_{q}q^{e_q}, f=\displaystyle\prod_{q}q^{r_q}$, then we have the following factorization
\[
K(n, f)=\prod_{q}K_q(q^{e_q}, q^{r_q}).
\]
\end{proposition}
\begin{proof}
Write
\[
N:=nf^2.
\]
For each prime $q$ set
$N_q:=q^{e_q+2r_q}$ so that 
\[
N=\prod_q N_q.
\]

\medskip
By the Chinese remainder theorem there is a canonical isomorphism
\begin{equation}\label{eqn: local-factor-integer}
(\Z/N\Z)^2 \;\cong\; \prod_q (\Z/N_q\Z)^2.
\end{equation}
Therefore, summing over residue classes $(a,b)\bmod N$ is the same as summing over tuples
$\bigl((a_q,b_q)\bmod N_q\bigr)_q$.

By definition,
\[
K(n,f)
=
\sum_{(a,b)\,(\mathrm{mod}\,N)}
S(a,b;n,f).
\]
Lemma~\ref{lem: multiplicativity of S} gives
\[
S(a,b;n,f)=\prod_q S_q(a,b;e_q,r_q).
\]
By Corollary~\ref{cor: periodicity modulo v+2r}, the \(q\)-factor
depends only on \((a,b)\bmod N_q\).  The Chinese remainder
identification \eqref{eqn: local-factor-integer} therefore gives
\[
K(n,f)
=
\sum_{\bigl((a_q,b_q)\bmod N_q\bigr)_q}\ \ 
\ \prod_q
S_q(a_q, b_q; e_q, r_q),
\]
Since all sums are finite and the summand is a product of functions of the individual prime $q$,
we can factor the total sum as a product of local sums:
\[
K(n,f)
=
\prod_q
\left(
\sum_{(a_q,b_q)\,(\mathrm{mod}\,N_q)}
S_q(a_q, b_q; e_q, r_q)
\right).
\]
By the definition \eqref{eqn: local-kloosterman-p}, the bracketed factor is precisely $K_q(q^{e_q},q^{r_q})$.
This proves the claimed factorization.
\end{proof}


\subsection{Local constancy of Kloosterman sums}\label{subsec: periodicity}

This subsection will be devoted to the proof of the Theorem \ref{thm: periodicity}. Our proof is local and consists of proving an analogous result for its local Euler factors.

\begin{lemma}\label{lem:overorders-I}
Let \(R_0\) be a finite free \(\Z\)-algebra, put
\(A_0:=R_0\otimes_\Z\Q\), and let \(f\ge1\).  For an over-ring
\(R_0\subseteq R\subseteq A_0\) with \([R:R_0]=f\), set
\(I:=fR\subseteq R_0\).
Then $I$ satisfies
\[
fR_0\subseteq I\subseteq R_0,\qquad I^2\subseteq fI,\qquad [I: fR_0]=f
\]
and
\[
R=\{x\in (1/f)R_0:\ fx\in I\}.
\]
Conversely, any $R_0$-submodule $I$ with $fR_0\subseteq I\subseteq R_0$, $I^2\subseteq fI$ and $[I: fR_0]=f$
arises uniquely from an over-ring $R$ with $[R:R_0]=f$ via $I=fR$.

The same statement holds for a finite free \(\Z_q\)-algebra after
replacing \(f\) by \(q^r\).
\end{lemma}
\begin{proof}
The local and global proofs are identical, so we prove the global
statement. If $R\supseteq R_0$ and $[R:R_0]=f$ then $fR\subseteq R_0$ (since $R/R_0$ has exponent dividing $f$),
so $I=fR$ is an $R_0$-submodule between $fR_0$ and $R_0$.  Moreover $I^2=(fR)^2=f^2R^2\subseteq f(fR)=fI$.
Multiplication by \(f\) induces an additive isomorphism
\(R/R_0\cong fR/fR_0\), and therefore
\[
[I: f R_0]=[R:R_0]=f.
\]

Conversely, given $I$ with $fR_0\subseteq I\subseteq R_0$ and $I^2\subseteq fI$, define
$R:=\{x\in (1/f)R_0:\ fx\in I\}$.  Then $R$ is a ring containing $R_0$:
if $x,y\in R$ with $fx=u\in I$ and $fy=v\in I$, then
\[
f(xy)=\frac{uv}{f}\in \frac{I^2}{f}\subseteq I,
\]
so $xy\in R$.  Also $fR=I$ by construction, and the additive bijection $R/R_0\cong I/fR_0$
shows $[R:R_0]=[I:fR_0]=f$.
\end{proof}

\begin{lemma}\label{lem: strong local periodicity}
Let \(q\) be a prime and \(v,r\ge0\).  If \(v\ge1\), then
\[
S_q(a,b;v,r)
\quad\text{depends only on}\quad
(a,b)\pmod{q^{\,1+2r}}.
\]
If \(v=0\), it depends only on \((a,b)\pmod{q^{\,2r}}\).
\end{lemma}
\begin{proof}
Write
\[
R_0:=R^\epsilon(a,b),\qquad R_0':=R^\epsilon(a',b').
\]
Whenever \((a,b)\equiv(a',b')\pmod{q^m}\), the two monic
presentations give a canonical ring isomorphism
\[
\varphi_m:R_0/q^mR_0\xrightarrow{\ \sim\ }R_0'/q^mR_0',
\qquad X\longmapsto X.
\]
By Lemma~\ref{lem:overorders-I}, an index-\(q^r\) over-ring \(R\)
of \(R_0\) corresponds to
\(I:=q^rR\subseteq R_0\), where
\[
q^rR_0\subseteq I\subseteq R_0,\qquad
I^2\subseteq q^rI,\qquad
[I:q^rR_0]=q^r.
\]
Assume first that the coefficients are congruent modulo \(q^{2r}\).
For \(r>0\), define \(I'\subseteq R_0'\) to be the inverse image of
\(\varphi_{2r}(I/q^{2r}R_0)\); for \(r=0\), put \(I'=R_0'\).
The displayed module, index, and multiplication conditions are all
determined modulo \(q^{2r}\), so \(I\mapsto I'\) is a bijection.
Applying Lemma~\ref{lem:overorders-I} again gives a corresponding
index-\(q^r\) over-ring \(R'\) of \(R_0'\).

Suppose now that the coefficients are congruent modulo \(q^{2r+1}\),
and let \(R\leftrightarrow R'\) be corresponding over-rings, with
\(I=q^rR\) and \(I'=q^rR'\).  Multiplication by \(q^r\) identifies
the additive groups
\[
R/qR\cong I/qI,\qquad R'/qR'\cong I'/qI'.
\]
On \(I/qI\), the transported product is
\[
u\star v:=\frac{uv}{q^r}\pmod{qI}.
\]
It is determined by \(uv\) modulo \(q^{r+1}I\), and
\(q^{r+1}I\supseteq q^{2r+1}R_0\).  Thus the isomorphism
\[
\varphi_{2r+1}:R_0/q^{2r+1}R_0
\xrightarrow{\ \sim\ }R_0'/q^{2r+1}R_0'
\]
is compatible with \(\varphi_{2r}\), carries \(I\) to \(I'\) at this
level, and identifies the transported products.  Consequently,
\[
R/qR\cong R'/qR'.
\]

By \eqref{eq:def-zetaS-local}, all coefficients \(W_R(q^\nu)\) depend
only on the residue fields of \(R/qR\).  The Gorenstein property also
depends only on this finite ring: for a finite flat \(\Z_q\)-algebra
\(A\), the \(A\)-module \(A^\vee=\Hom_{\Z_q}(A,\Z_q)\) is cyclic if
and only if \(A^\vee/qA^\vee\cong\Hom_{\F_q}(A/qA,\F_q)\) is cyclic
over \(A/qA\); see
\cite[\href{https://stacks.math.columbia.edu/tag/04GM}{Tag 0E9M}]{stacks}.
Consequently, corresponding over-rings have the same
\(W_R(q^\nu)\) and the same \(a_R(q)\).  Since
\(a_R(q^t)=0\) for \(t\ge2\), each term in \(S_q(a,b;v,r)\) is
preserved when \(v\ge1\).

When \(v=0\), the summand is \(a_R(1)W_R(1)=1\), so
\(S_q(a,b;0,r)\) simply counts the index-\(q^r\) over-rings.
The first paragraph shows that this count depends only modulo
\(q^{2r}\).
\end{proof}

\begin{corollary}\label{cor: periodicity modulo v+2r}
For all \(v,r\ge0\), the sum \(S_q(a,b;v,r)\) depends only on
\((a,b)\pmod{q^{v+2r}}\).
\end{corollary}
\begin{proof}
If \(v\ge1\), congruence modulo \(q^{v+2r}\) implies congruence
modulo \(q^{1+2r}\), and the result follows from
Lemma~\ref{lem: strong local periodicity}.  For \(v=0\), the two
moduli are both \(q^{2r}\).
\end{proof}

From this we can finally prove Theorem \ref{thm: periodicity}.
\begin{proof}[Proof of Theorem \ref{thm: periodicity}]
    The sum in Theorem \ref{thm: periodicity} is $S(a, b; n, f)$, whose local factors are $S_q(a,b;v_q(n),v_q(f))$. We have $(a, b)\equiv (a', b')\pmod{nf^2}$ if and only if, for each prime $q$, $(a, b)\equiv (a', b')\pmod{q^{v_q(n)+2v_q(f)}}$. Corollary \ref{cor: periodicity modulo v+2r} implies that the local factors are periodic modulo $q^{v_q(n)+2v_q(f)}$, and the result follows.
\end{proof}

Finally, we conclude with a further simplifaction of the Kloosterman sums that simplifies the computations of their explicit values. We have

\begin{proposition}\label{prop: simplification-local-kloosterman}
    Let $q$ be a prime, then
    \begin{equation}\label{eqn: local-kloosterman-p-new} 
        K_q(q^v, q^r) = q^{2(v - 1)}\sum_{(a,b)\,(\mathrm{mod}\,q^{1+2r})}S_q(a,b; v, r).
    \end{equation}
\end{proposition}
\begin{proof}
If $v\ge 1$, Lemma~\ref{lem: strong local periodicity} shows that
\(S_q(a,b;v,r)\) depends only on
\[
(a,b)\pmod{q^{1+2r}}.
\]
Each residue class modulo \(q^{1+2r}\) has exactly \(q^{2(v-1)}\) lifts modulo \(q^{v+2r}\). Therefore
\begin{align*}
K_q(q^v,q^r)
&=\sum_{(a,b)\,(\mathrm{mod}\,q^{v+2r})}S_q(a,b;v,r)\\
&=q^{2(v-1)}\sum_{(a,b)\,(\mathrm{mod}\,q^{1+2r})}S_q(a,b;v,r),
\end{align*}
which proves \eqref{eqn: local-kloosterman-p-new} in this case.

For $v=0$, the last paragraph of the proof of
Lemma~\ref{lem: strong local periodicity} shows that
$S_q(a,b;0,r)$ depends only on $(a,b)\pmod{q^{2r}}$.
Each residue class modulo \(q^{2r}\) has exactly \(q^2\) lifts modulo \(q^{1+2r}\), hence
\begin{align*}
K_q(1,q^r)
&=\sum_{(a,b)\,(\mathrm{mod}\,q^{2r})}S_q(a,b; 0, r)\\
&=q^{-2}\sum_{(a,b)\,(\mathrm{mod}\,q^{1+2r})}S_q(a,b; 0, r),
\end{align*}
which is again \eqref{eqn: local-kloosterman-p-new} for \(v=0\).
\end{proof}
\begin{remark}
    The explicit computation of the Kloosterman sums can now be organized by first working modulo $q^{1+2r}$ and then accounting for the extra lifts through the factor $q^{2(v-1)}$. In particular, for $v\ge 1$ the dependence on the coefficient space can be reduced to residue classes modulo $q^{1+2r}$; see \cite[Section~10]{DengEspinosaYunZeta} for the resulting uniform computation.
\end{remark}
\section{The Analysis of Residues and the Isolation of the Traces}

Recall that by definition
\begin{equation*}
   \Theta_{n, f}^{\pm}(x, y) = \Phi^{\pm}_{n, f}(x, y) + nf^2\Psi^{\pm}_{n, f}(x, y)
\end{equation*}
Therefore, it makes sense to separate our study into the study of two corresponding sums. 

\begin{definition}
\label{def: sums R1 and R2}
We define the sums
\begin{align*}
    R_1 &= 
        \displaystyle\sum_{\pm}
        \sum_{n=1}^{\infty}
        \sum_{f=1}^{\infty}\dfrac{K(n, f)}{n^3f^5}
        \displaystyle\int_{\R^2}\Phi_{n, f}^{\pm}(x, y)dxdy,\\
    R_2 &= 
        \displaystyle\sum_{\pm}
        \sum_{n=1}^{\infty}
        \sum_{f=1}^{\infty}\dfrac{K(n, f)}{n^3f^5}
        \displaystyle\int_{\R^2}nf^2\Psi_{n, f}^{\pm}(x, y)dxdy.
\end{align*}
\end{definition}
The analyses are similar, but the explicit final results differ because of the definitions of the involved functions. For $R_1$ we will deal with the function $F$, while for $R_2$ we must deal with $H_{(x, y)}$. The difference manifests in the poles and residues they have.

\subsection{Analysis of $R_1$}

We first analyze \(R_1\).  For \(\epsilon\in\{\pm1\}\), write

\[
 \Pol^\epsilon(x,y):=\Pol(x,y,\epsilon).
\]
In an unscaled \(\epsilon\)-summand below, the abbreviation
\(\Pol(x,y)\) continues to mean \(\Pol(x,y,\epsilon p^k)\); after the
coefficient normalization we use \(\Pol^\epsilon\).
By Theorem~\ref{thm: Dirichlet Sums Evaluation}, the two functions
\(\D^+\) and \(\D^-\) have the same meromorphic continuation; denote it
by \(\D\).
For a meromorphic function \(A\), we use the abbreviation
\[
 \int_c A(s)\,ds
 :=\frac{1}{2\pi i}\int_{\Re(s)=c}A(s)\,ds,
\]
with the vertical line oriented upward.

\begin{lemma}[Integrability along the cubic discriminant]
\label{lem:cubic-discriminant-integrability}
For each \(\epsilon\in\{\pm1\}\), the integral
\begin{equation}
\label{eq:cubic-discriminant-Mellin-integral}
 A_\epsilon(w):=
 \int_{\R^2}\theta^\epsilon(\boldsymbol b)
       |\Pol^\epsilon(\boldsymbol b)|^w\,d\boldsymbol b
\end{equation}
is absolutely convergent and holomorphic for \(\Re(w)>-5/6\).
\end{lemma}

\begin{proof}
The normalized orbital integral has compact support.  The real
Shalika-germ expansion also gives, locally along the discriminant,
\[
 |\theta^\epsilon(\boldsymbol b)|
 \ll \bigl(1+|\log|\Pol^\epsilon(\boldsymbol b)||\bigr)^N
\]
for some \(N\); see \cite[Sections~2--3]{ShelstadGerms1979} and the
application to \(\GL_n\) in \cite[Section~3]{GKMPW}.

It therefore remains to determine the local integrability exponent of
the discriminant.  At every smooth point of its zero set, the
discriminant itself is a local coordinate, so the required condition is
\(\Re(w)>-1\).  At a triple-root point, depress the cubic
\(X^3-aX^2+bX-\epsilon\) and use the local coordinates
\[
 A=b-\frac{a^2}{3},
 \qquad
 B=-\frac{2a^3}{27}+\frac{ab}{3}-\epsilon.
\]
If the cubic is \((X-r)^3\), then \(r^3=\epsilon\),
\((a,b)=(3r,3r^2)\), and the Jacobian determinant of
\((a,b)\mapsto(A,B)\) is \(-r^2\ne0\).  Moreover,
\[
 \Pol^\epsilon=-4A^3-27B^2.
\]
On each of the regions \(A>0\) and \(A<0\), put
\(A=\pm\rho^2\) and \(B=\rho^3v\).  Up to a nonzero constant,
\[
 dA\,dB=\rho^4\,d\rho\,dv,
 \qquad
 |\Pol^\epsilon|=\rho^6|4\pm27v^2|.
\]
The radial integral is locally finite precisely when
\(4+6\Re(w)>-1\), while the possible simple zeros in the \(v\)-factor
require only \(\Re(w)>-1\).  Hence the cusp threshold is
\(\Re(w)>-5/6\).  Powers of the logarithm do not change either
threshold.  The same estimates after differentiation with respect to
\(w\) prove holomorphy in this half-plane by dominated convergence.
\end{proof}

For \(a\in\{2,3\}\), put
\begin{equation}
\label{eq:kappa-a-GL3}
 \kappa_{a,k,p}:=\operatorname*{Res}_{u=1/a}\D(u)
 =\frac1a
   \frac{\displaystyle
      \prod_{\substack{m\in\{2,3\}\\m\ne a}}\zeta(m/a)}
        {\zeta(1+1/a)\zeta(1+2/a)}
   \prod_{j=1}^{2}
      \frac{1-p^{-(k+j)/a}}{1-p^{-j/a}}.
\end{equation}
In particular,
\[
 \kappa_{2,k,p}
 =\frac{1}{2\zeta(2)}
   \prod_{j=1}^{2}
   \frac{1-p^{-(k+j)/2}}{1-p^{-j/2}}.
\]
Define
\begin{align}
\label{eq:R1-midpoint-residue}
 \mathcal R_{1,-1/2}(f^{p,k})
 :={}&\kappa_{2,k,p}\widetilde F(-1/2)p^{k(1-\alpha)}
 \sum_{\epsilon=\pm1}\int_{\R^2}
       \theta^\epsilon(\boldsymbol b)
       |\Pol^\epsilon(\boldsymbol b)|^{-\alpha/2}
       \,d\boldsymbol b.
\end{align}

We shall also use the next residue,
\begin{equation}
\label{eq:R1-minus-two-thirds-residue}
 \mathcal R_{1,-2/3}(f^{p,k})
 :=\kappa_{3,k,p}\widetilde F(-2/3)p^{k(1-4\alpha/3)}
 \sum_{\epsilon=\pm1}\int_{\R^2}
       \theta^\epsilon(\boldsymbol b)
       |\Pol^\epsilon(\boldsymbol b)|^{-2\alpha/3}
       \,d\boldsymbol b.
\end{equation}

\begin{proposition}
    \label{prop: Finding trivial rep}
Choose \(\sigma_1\) with
\(-5/6<\sigma_1<-2/3\).  Then \(R_1\) equals the sum of
\begin{equation}
\label{eqn: R1 residue at 0}
       p^k
        \dfrac{1 - p^{-(k+1)}}{1 - p^{-1}} \cdot \dfrac{1 - p^{-(k+2)}}{1 - p^{-2}}\displaystyle\sum_{\pm}
        \displaystyle\int_{\R^2}\theta^{\pm}(x, y)      
         dxdy,
\end{equation}
the residues \(\mathcal R_{1,-1/2}(f^{p,k})\) and
\(\mathcal R_{1,-2/3}(f^{p,k})\), and
\begin{align}
\label{eqn: R1 contour at negative sigma}
 \mathcal Q_1(f^{p,k})
 &:={}\displaystyle\sum_{\pm}\displaystyle\int_{\sigma_1}
        \displaystyle\int_{\R^2}\theta^{\pm}(xp^{-k/3}, yp^{-2k/3})
        \widetilde{F}(z)
        |\Pol(x, y)|^{\alpha z}
         \frac{\zeta(2z+2)}{\zeta(z+2)}
         \frac{\zeta(3z+3)}{\zeta(2z+3)}\notag\\
 &\hspace{25mm}\times
         \frac{1-p^{-(z+1)(k+1)}}{1-p^{-(z+1)}}
         \frac{1-p^{-(z+1)(k+2)}}{1-p^{-2(z+1)}}
         \,dx\,dy\,dz.
\end{align}
Equivalently,
\begin{equation}
\label{eq:R1-three-residues}
 R_1=\Tr\!\left(\mathbf1(f^{p,k})\right)
     +\mathcal R_{1,-1/2}(f^{p,k})
     +\mathcal R_{1,-2/3}(f^{p,k})
     +\mathcal Q_1(f^{p,k}).
\end{equation}
\end{proposition}
\begin{proof}
We have
\begin{align*}
 R_1    &= 
        \displaystyle\sum_{\pm}
        \sum_{n=1}^{\infty}
        \sum_{f=1}^{\infty}\dfrac{K(n, f)}{n^3f^5}
        \displaystyle\int_{\R^2}\Phi_{n, f}^{\pm}(x, y)dxdy,\\
        &=
        \displaystyle\sum_{\pm}
        \sum_{n=1}^{\infty}
        \sum_{f=1}^{\infty}\dfrac{K(n, f)}{n^3f^5}
        \displaystyle\int_{\R^2}\theta^{\pm}(xp^{-k/3}, yp^{-2k/3})F\left(\dfrac{nf^2}{|\Pol(x, y)|^{\alpha}}\right)dxdy\\
        &=
        \displaystyle\sum_{\pm}
        \sum_{n=1}^{\infty}
        \sum_{f=1}^{\infty}\dfrac{K(n, f)}{n^3f^5}
        \displaystyle\int_{\R^2}\theta^{\pm}(xp^{-k/3}, yp^{-2k/3})\displaystyle\int_{\lambda} \widetilde{F}(z)\left(\dfrac{nf^2}{|\Pol(x, y)|^{\alpha}}\right)^{-z}dzdxdy
\end{align*}
Initially the Mellin integral is over a line \(\Re(z)=\lambda>1\).
On that line the sums and integrals are absolutely convergent, so their
order may be exchanged and the Mellin integral may be brought outside.
We get
\begin{align*}
 R_1    &=
        \displaystyle\sum_{\pm}\displaystyle\int_{\lambda}
        \sum_{n=1}^{\infty}
        \sum_{f=1}^{\infty}\dfrac{K(n, f)}{n^3f^5}
        \displaystyle\int_{\R^2}\theta^{\pm}(xp^{-k/3}, yp^{-2k/3}) \widetilde{F}(z)\dfrac{|\Pol(x, y)|^{\alpha z}}{n^zf^{2z}}dxdy dz\\
        &=
        \displaystyle\sum_{\pm}\displaystyle\int_{\lambda}
        \displaystyle\int_{\R^2}\theta^{\pm}(xp^{-k/3}, yp^{-2k/3})
        \widetilde{F}(z)
        |\Pol(x, y)|^{\alpha z}
        \sum_{n=1}^{\infty}
        \sum_{f=1}^{\infty}\dfrac{K(n, f)}{n^{z + 3}f^{2z + 5}}
         dxdy dz\\
         &=
        \displaystyle\sum_{\pm}\displaystyle\int_{\lambda}
        \displaystyle\int_{\R^2}\theta^{\pm}(xp^{-k/3}, yp^{-2k/3})
        \widetilde{F}(z)
        |\Pol(x, y)|^{\alpha z}
        \sum_{n=1}^{\infty}
        \sum_{f=1}^{\infty}\dfrac{K(n, f)}{n^{(z + 1) + 2}f^{2(z + 1) + 3}}
         dxdy dz\\
          &=
        \displaystyle\sum_{\pm}\displaystyle\int_{\lambda}
        \displaystyle\int_{\R^2}\theta^{\pm}(xp^{-k/3}, yp^{-2k/3})
        \widetilde{F}(z)
        |\Pol(x, y)|^{\alpha z}
        \D^{\pm}(z + 1)dxdydz
\end{align*}
Using Theorem~\ref{thm: Dirichlet Sums Evaluation} on this initial
line gives
\begin{equation*}
        \displaystyle\sum_{\pm}\displaystyle\int_{\lambda}
        \displaystyle\int_{\R^2}\theta^{\pm}(xp^{-k/3}, yp^{-2k/3})
        \widetilde{F}(z)
        |\Pol(x, y)|^{\alpha z}
         \cdot\dfrac{\zeta(2z+2)}{\zeta(z+2)} \cdot \dfrac{\zeta(3z + 3)}{\zeta(2z+3)} \cdot \dfrac{1 - p^{-(z+1)(k+1)}}{1 - p^{-(z+1)}} \cdot \dfrac{1 - p^{-(z+1)(k+2)}}{1 - p^{-2(z+1)}}dxdydz.
\end{equation*}

We now continue this Mellin integrand meromorphically and move the
contour to \(\Re(z)=\sigma_1\).  The poles crossed are \(z=0\), from
\(\widetilde F(z)\), \(z=-1/2\), from \(\zeta(2z+2)\), and
\(z=-2/3\), from \(\zeta(3z+3)\).  There is no denominator-zero
obstruction in this strip: the choice \(\sigma_1>-5/6\) gives
\[
 \Re(z+2)>7/6,
 \qquad
 \Re(2z+3)>4/3,
\]
so both denominator zeta functions are zero-free.

The residue at \(z=0\) is \eqref{eqn: R1 residue at 0}; changing
variables in its coefficient integral and applying
Proposition~\ref{prop: value trivial} identifies it with
\(\Tr(\mathbf1(f^{p,k}))\).  At \(z=-1/2\), one has
\[
 {\rm Res}_{z=-1/2}\D(z+1)=\kappa_{2,k,p}.
\]
The change of variables
\(x=p^{k/3}b_1\), \(y=p^{2k/3}b_2\) gives
\[
 dx\,dy=p^k\,d\boldsymbol b,
 \qquad
 |\Pol(x,y,\epsilon p^k)|
 =p^{2k}|\Pol^\epsilon(\boldsymbol b)|.
\]
Thus the residue at \(z=-1/2\) is precisely
\(\mathcal R_{1,-1/2}(f^{p,k})\).  At the next pole,
\[
 {\rm Res}_{z=-2/3}\D(z+1)=\kappa_{3,k,p}.
\]
The same change of variables gives
\[
 dx\,dy\,|\Pol(x,y,\epsilon p^k)|^{-2\alpha/3}
 =p^{k(1-4\alpha/3)}
  \,d\boldsymbol b\,
  |\Pol^\epsilon(\boldsymbol b)|^{-2\alpha/3}.
\]
Hence the residue at \(z=-2/3\) is
\(\mathcal R_{1,-2/3}(f^{p,k})\).  The decay of
\(\widetilde F\) on vertical lines and
Lemma~\ref{lem:cubic-discriminant-integrability} justify the contour
shift, because \(\alpha\sigma_1>\sigma_1>-5/6\).  This proves
\eqref{eq:R1-three-residues}.
\end{proof}

\subsection{Analysis of $R_2$}

We now want to do a similar analysis for $R_2$. 

\begin{proposition}
    \label{prop: residues of H}
    Fix a sign $\pm$. Define $C(x, y)$ as the indicator function of the subset $D^{\pm}$ on which $\Pol(x, y, \pm 1) > 0$. For each $(x, y)\in D^{\pm}$ we have
    \begin{equation*}
        \displaystyle\lim_{s \longrightarrow 0} \dfrac{1 - p^{-s(k+1)}}{1 - p^{-s}} \cdot \dfrac{1 - p^{-s(k+2)}}{1 - p^{-2s}}\cdot \dfrac{\zeta(2s)}{\zeta(s+1)} \cdot \dfrac{\zeta(3s)}{\zeta(2s+1)}\cdot H_{(x, y)}(s) = (k + 1)(k + 2)\sqrt{\pi}\cdot C(x, y).
    \end{equation*}
\end{proposition}
\begin{proof}
Notice that
\begin{equation*}
 \dfrac{1 - p^{-s(k+1)}}{1 - p^{-s}} \cdot \dfrac{1 - p^{-s(k+2)}}{1 - p^{-2s}} = \dfrac{1 - p^{-s(k+1)}}{1 - p^{-s}} \cdot \dfrac{1 - p^{-s(k+2)}}{1 - p^{-s}} \dfrac{1}{1 + p^{-s}}. 
\end{equation*}
Thus
\begin{equation*}
   \displaystyle\lim_{s\longrightarrow 0}\dfrac{1 - p^{-s(k+1)}}{1 - p^{-s}} \cdot \dfrac{1 - p^{-s(k+2)}}{1 - p^{-2s}} = \dfrac{(k + 1)(k + 2)}{2}. 
\end{equation*}
We conclude this part of the product can always be evaluated at $s = 0$.

We now separate into cases. Suppose first that
\(\Pol^\pm(x,y)<0\). Reorganizing the involved function we have
    \begin{equation*}
     \dfrac{\zeta(2s)}{\zeta(s+1)} \cdot \dfrac{\zeta(3s)}{\zeta(2s+1)}\cdot H_{(x, y)}(s) =   \pi^{1/2(1 - 2s)}(2\pi)^{1-2s}\cdot \zeta(2s)\zeta(3s) \cdot\dfrac{\Gamma(s)}{\zeta(s + 1)} \cdot \dfrac{1}{\zeta(2s + 1)} \cdot  \dfrac{1}{\Gamma(1 - s)}.
    \end{equation*}
    Recall that \(\zeta(0)=-1/2\) and \(\Gamma(1)=1\). Furthermore, because $\Gamma(s)$ and $\zeta(s + 1)$ both have simple poles at $s = 0$ their quotient is well defined at $0$.  Finally, $\zeta(2s + 1)$ has a simple pole at $s = 0$, thus its inverse has simple zero there.  In conclusion, all appropriate limits exist at $0$ and we get
    \begin{equation*}
        \lim_{s \longrightarrow 0} \dfrac{\zeta(2s)}{\zeta(s+1)} \cdot \dfrac{\zeta(3s)}{\zeta(2s+1)}\cdot H_{(x, y)}(s) = 0.
    \end{equation*}
    We now suppose \(\Pol^\pm(x,y)>0\). A calculation with Laurent series gives
    \begin{equation*}
        \dfrac{\Gamma(s/2)\Gamma(s/2)}{\zeta(s + 1)\zeta(2s + 1)} 
        = \frac{4/s^2+O(1/s)}{1/(2s^2)+O(1/s)}=8+O(s)
    \end{equation*}
    Here
    \(\Gamma(s/2)=2/s+O(1)\),
    \(\zeta(s+1)=1/s+O(1)\), and
    \(\zeta(2s+1)=1/(2s)+O(1)\). Thus, taking the limit we get
    \begin{equation*}
      \displaystyle\lim_{s \longrightarrow 0} \dfrac{\zeta(2s)}{\zeta(s+1)} \cdot \dfrac{\zeta(3s)}{\zeta(2s+1)} \cdot  \pi^{3/2(1 - 2s)}\dfrac{\Gamma(s/2)^2}{\Gamma(\frac{1 - s}{2})^2} = \zeta(0)\zeta(0) \cdot \pi^{3/2} \cdot 8 \cdot \dfrac{1}{(\sqrt{\pi})^2} = 2\sqrt{\pi}
    \end{equation*}
    This concludes the proof.
\end{proof}

We now return to \(R_2\).  The positive poles of the Kloosterman
Dirichlet series must be recorded before the residue at the origin is
extracted.  Write \(\beta:=1-\alpha\),
and let \(H^\pm_{(x,y)}\) denote the function in
Definition~\ref{def: function H(x, y)}, with its case determined by the
sign of \(\Pol^\pm(x,y)\).  Since Theorem~\ref{thm: Dirichlet Sums Evaluation}
shows that \(\D^+=\D^-\), the common function is the \(\D\) introduced
above.  After the change of variables
\(x\mapsto p^{k/3}x\), \(y\mapsto p^{2k/3}y\), the Mellin expression for
\(R_2\) becomes
\begin{equation}
\label{eq:R2-single-Mellin-integral}
   R_2=\frac{1}{2\pi i}\int_{\Re(u)=\lambda}
          \mathcal I_2(u;f^{p,k})\,du,
   \qquad \lambda>1,
\end{equation}
where
\begin{align}
\label{eq:R2-Mellin-integrand}
\mathcal I_2(u;f^{p,k})
:={}&
 \pi^{u-1/2}\D(u)\widetilde F(u)p^{2k\beta u}
 \sum_{\pm}\int_{\R^2}
     \theta^\pm(x,y)H^\pm_{(x,y)}(u)
     |\Pol^\pm(x,y)|^{-1/2+\beta u}\,dx\,dy.
\end{align}

For \(a\in\{2,3\}\), define
\begin{equation}
\label{eq:positive-residue-Pa-GL3}
 \mathcal P_{a,\alpha}(f^{p,k})
 :=\kappa_{a,k,p}\widetilde F(1/a)p^{2k\beta/a}
 \sum_{\pm}\pi^{1/a-1/2}
 \int_{\R^2}\theta^\pm(x,y)H^\pm_{(x,y)}(1/a)
 |\Pol^\pm(x,y)|^{-1/2+\beta/a}\,dx\,dy.
\end{equation}

Choose \(1/3<\eta<1/2\).  Next choose \(\varepsilon,T>0\) sufficiently
small that \(1/2+\beta\varepsilon<5/6\) and the closed rectangle
\[
 -\varepsilon\leq\Re(u)\leq\eta,
 \qquad |\Im(u)|\leq T,
\]
contains no singularities of \(\mathcal I_2(u;f^{p,k})\) other than
\(u=0\) and \(u=1/3\), and has no singularity on its boundary.  Let
\(\Gamma_{\eta,\varepsilon,T}\) be the upward-oriented contour obtained
from the line \(\Re(u)=\eta\) by replacing its segment between
\(\eta-iT\) and \(\eta+iT\) by
\[
 [\eta-iT,-\varepsilon-iT],\quad
 [-\varepsilon-iT,-\varepsilon+iT],\quad
 [-\varepsilon+iT,\eta+iT].
\]
Set
\begin{equation}
\label{eq:R2-indented-remainder-GL3}
 \mathcal Q_2(f^{p,k})
 :=\frac{1}{2\pi i}
   \int_{\Gamma_{\eta,\varepsilon,T}}
      \mathcal I_2(u;f^{p,k})\,du.
\end{equation}

\begin{remark}
The indentation in \eqref{eq:R2-indented-remainder-GL3} is necessary.
A shift to an entire line \(\Re(u)<0\) may cross poles at
\[
 u=\rho-1\qquad\hbox{or}\qquad u=\frac{\rho-1}{2},
\]
where \(\rho\) is a nontrivial zero of the Riemann zeta function.  Such
poles arise from the reciprocal factors
\(\zeta(1+u)^{-1}\) and \(\zeta(1+2u)^{-1}\).
\end{remark}

\begin{proposition}
\label{prop:correct-R2-residue-decomposition}
One has
\begin{equation}
\label{eq:correct-R2-residue-decomposition}
 R_2=
 \mathcal P_{2,\alpha}(f^{p,k})
 +\mathcal P_{3,\alpha}(f^{p,k})
 +\frac13\Tr\!\left(\xi(f^{p,k})\right)
 +\mathcal Q_2(f^{p,k}).
\end{equation}
\end{proposition}

\begin{proof}
The condition on \(\varepsilon\) and
Lemma~\ref{lem:cubic-discriminant-integrability} give absolute
convergence and holomorphic dependence of the coefficient integrals
throughout the rectangle, since
\[
 \Re\!\left(-\frac12+\beta u\right)
 \geq-\frac12-\beta\varepsilon>-\frac56.
\]
For \(\Re(u)>0\), the factors
\(\zeta(1+u)\) and \(\zeta(1+2u)\) are nonzero, the apparent
singularities in the product of the two finite \(p\)-quotients are
removable (the product is a Gaussian polynomial in \(p^{-u}\)), and
\(H^\pm_{(x,y)}(u)\) and \(\widetilde F(u)\) are holomorphic away from
\(u=0\).  It follows from
\[
 \D(u)=
 \frac{\zeta(2u)\zeta(3u)}
      {\zeta(1+u)\zeta(1+2u)}
 \prod_{j=1}^{2}
      \frac{1-p^{-u(k+j)}}{1-p^{-ju}}
\]
that the only pole crossed in moving the line in
\eqref{eq:R2-single-Mellin-integral} from \(\Re(u)=\lambda\) to
\(\Re(u)=\eta\) is the simple pole at \(u=1/2\), whose residue is
\(\mathcal P_{2,\alpha}(f^{p,k})\).

Proposition~\ref{prop: residues of H} computes the remaining residue at
the origin as
\begin{equation}
\label{eqn: residue R2}
 (k+1)(k+2)\sum_{\pm}
 \int_{\R^2}\theta^\pm(x,y)|\Pol^\pm(x,y)|^{-1/2}
 C(x,y)\,dx\,dy.
\end{equation}
By Proposition~\ref{prop: value of Eisenstein term}, this is
\(\frac13\Tr(\xi(f^{p,k}))\).  The difference between the upward line
\(\Re(u)=\eta\) and \(\Gamma_{\eta,\varepsilon,T}\) is the positively
oriented boundary of the displayed rectangle.  Its enclosed poles are
\(u=1/3\), whose residue is
\(\mathcal P_{3,\alpha}(f^{p,k})\), and \(u=0\), whose residue is the
trace term just computed.  The residue theorem therefore gives
\eqref{eq:correct-R2-residue-decomposition}.

The horizontal integrals used in the first contour shift tend to zero by
the decay of \(\widetilde F\), the standard vertical-strip estimates for
the zeta function, and Stirling's formula applied to the explicit gamma
quotients in Definition~\ref{def: function H(x, y)}.  These estimates
also justify the preceding changes in the order of integration.
\end{proof}

\begin{lemma}
\label{lem:midpoint-cancellation-GL3}
The residues at \(z=-1/2\) in \(R_1\) and at \(u=1/2\) in \(R_2\)
cancel exactly:
\begin{equation}
\label{eq:midpoint-cancellation-GL3}
 \mathcal R_{1,-1/2}(f^{p,k})
 +\mathcal P_{2,\alpha}(f^{p,k})=0.
\end{equation}
\end{lemma}

\begin{proof}
The corrected gamma quotient in
Definition~\ref{def: function H(x, y)} gives
\[
 H^\pm_{(x,y)}(1/2)=1
\]
for both real signatures.  Moreover,
\[
 \pi^{1/2-1/2}=1,
 \qquad
 -\frac12+\frac{1-\alpha}{2}=-\frac{\alpha}{2}.
\]
Consequently, \eqref{eq:positive-residue-Pa-GL3} with \(a=2\) becomes
\begin{align*}
 \mathcal P_{2,\alpha}(f^{p,k})
 ={}&\kappa_{2,k,p}\widetilde F(1/2)p^{k(1-\alpha)}
 \sum_{\epsilon=\pm1}\int_{\R^2}
       \theta^\epsilon(\boldsymbol b)
       |\Pol^\epsilon(\boldsymbol b)|^{-\alpha/2}
       \,d\boldsymbol b.
\end{align*}
This is the negative of \eqref{eq:R1-midpoint-residue}, because
\(\widetilde F(-s)=-\widetilde F(s)\) by
Lemma~\ref{lem: uniform-bound-tF}.  Both residues enter their respective
leftward contour shifts with a plus sign; the cancellation comes from
the oddness of \(\widetilde F\), not from contour orientation.
\end{proof}

\begin{remark}
The cancellation in Lemma~\ref{lem:midpoint-cancellation-GL3} is the
same mechanism as in Altu\u{g}'s \(\GL(2)\) zero-frequency calculation
\cite[ Theorem~6.1]{AliI}.  In both cases, the residue terms at
the paired Mellin points \(-1/2\) and \(1/2\) have the same remaining
factors and cancel because
\(\widetilde F(-1/2)=-\widetilde F(1/2)\).
\end{remark}

Combining Proposition~\ref{prop: Finding trivial rep} with
Lemma~\ref{lem:midpoint-cancellation-GL3} gives the useful form
\begin{equation}
\label{eq:R1-after-midpoint-cancellation}
 R_1=\Tr\!\left(\mathbf1(f^{p,k})\right)
     -\mathcal P_{2,\alpha}(f^{p,k})
     +\mathcal R_{1,-2/3}(f^{p,k})
     +\mathcal Q_1(f^{p,k}).
\end{equation}

After the  cancellation, the remaining cubic-pole terms in
\(R_1+R_2\) are \(\mathcal R_{1,-2/3}(f^{p,k})\), defined in
\eqref{eq:R1-minus-two-thirds-residue}, and
\(\mathcal P_{3,\alpha}(f^{p,k})\), defined in
\eqref{eq:positive-residue-Pa-GL3}.  We abbreviate their sum by
\begin{equation}
\label{eq:combined-cubic-residue-definition}
 \mathcal C_{3,\alpha}(f^{p,k})
 :=
\mathcal R_{1,-2/3}(f^{p,k})
 +\mathcal P_{3,\alpha}(f^{p,k}).
\end{equation}

\begin{remark}
We do not presently know whether
\(\mathcal C_{3,\alpha}(f^{p,k})\) admits a spectral interpretation as
a trace distribution arising from explicitly described automorphic
representations.  Thus \eqref{eq:combined-cubic-residue-definition}
should be understood only as grouping the two cubic-pole residues.
\end{remark}

\subsection{Putting it all together}

The nonzero dual frequencies were collected in
\(\Sigma((\xi,\eta)\ne(0,0))\) in
\eqref{eq:nonzero-frequency-contribution}.

With this we are ready to state our final result.

\begin{theorem}
    \label{thm: trace is there}
One has
\begin{align}
\label{eq:correct-final-decomposition-GL3}
 \Iel(f^{p,k})
 ={}&\Tr\!\left(\mathbf1(f^{p,k})\right)
 +\frac13\Tr\!\left(\xi(f^{p,k})\right)
 +\mathcal C_{3,\alpha}(f^{p,k})\notag\\
 &+\mathcal Q_1(f^{p,k})
 +\mathcal Q_2(f^{p,k})
 +\Sigma((\xi,\eta)\ne(0,0))-\Sigma(f^{p,k}).
\end{align}
\end{theorem}
\begin{proof}
The zero-frequency contribution is \(R_1+R_2\).  Proposition
\ref{prop: Finding trivial rep} identifies \(R_1\) with the trace of the
trivial representation, \(\mathcal R_{1,-1/2}\),
\(\mathcal R_{1,-2/3}\), and \(\mathcal Q_1\).
Proposition~\ref{prop:correct-R2-residue-decomposition}
gives the four terms arising from \(R_2\).  Lemma
\ref{lem:midpoint-cancellation-GL3} cancels
\(\mathcal R_{1,-1/2}\) against \(\mathcal P_{2,\alpha}\).
Equation~\eqref{eq:combined-cubic-residue-definition} combines
\(\mathcal R_{1,-2/3}\) with \(\mathcal P_{3,\alpha}\) as
\(\mathcal C_{3,\alpha}\).
Substituting the resulting identity into the Poisson decomposition proves
\eqref{eq:correct-final-decomposition-GL3}.
\end{proof}

\bibliographystyle{acm}
\bibliography{Bibliography}

\end{document}